\documentclass[a4paper,12pt]{article}
\usepackage{amsmath} 
\usepackage{amsthm}
\usepackage{amssymb,bm,color,enumerate} 
\usepackage{eucal}%,mathscr}
\usepackage{geometry}
\usepackage{pstricks}
\usepackage{graphicx}
\usepackage{pst-node}

\usepackage[T1]{fontenc}
\usepackage[latin2]{inputenc}

\usepackage {tikz}
\usetikzlibrary {positioning}
\usepackage{hangcaption}
\numberwithin{equation}{section}
\theoremstyle{plain}
\newtheorem{theorem}{Theorem}[section]
\newtheorem{proposition}[theorem]{Proposition}
\newtheorem{corollary}[theorem]{Corollary}
\newtheorem{lemma}[theorem]{Lemma}
\theoremstyle{definition}
\newtheorem{example}[theorem]{Example}
\newtheorem{remark}[theorem]{Remark}
\newtheorem{definition}[theorem]{Definition}

\newcommand{\Node}[1]{\makebox[3mm]{#1}}
\def\cvput#1[#2]{\pnode(#1,1){#1} \pscircle*(#1,1){.1} \rput(#1,.5){$#2$}}

\def\C{{\mathbb C}}
\def\R{{\mathbb R}}
\def\N{{\mathbb N}}
\def\Z{{\mathbb Z}}

\def\d{{\rm d}}    % for measures and integration 
\def\grp{{\rm grp}}    % group 

\def\D{d_{q}}     % Annihilation operator of type D 
\def\G{G_{q}}    % Gaussian operator of type D
\def\ra{r}         % right annihilation operator 
\def\la{l}         % left annihilation operator 

\def\F{\mathcal{F}_{\rm fin}(H)}   % algebraic Fock space
\def\Fq{\mathcal{F}_{\rm fin}^{(q)}(H)}   % algebraic Fock space of type D
\def\H{H}        % Hilbert space
\def\A{\text{\rm vN}(G_{q}(H_\R))}

\def\PD{\mathcal{P}^D}    % Partitions of type D
\def\NC{\mathcal{NC}}
\def\P{\mathcal{P}}

\def\Cr{\text{\normalfont cr}}     %crossings
\def\lcr{\text{\normalfont lcr}}    % left crossings 
\def\cs{\text{\normalfont cs}}    % covered singletons
\def\cnp{\text{\normalfont cnp}}   % covered negative pairs 
\def\npssr{\text{\normalfont npssr}}  %negative pair and singleton on the strict right 
\def\cover{\text{\normalfont cover}}  % covers 
\def\np{\text{\normalfont np}}      % negative pairs 
\def\npr{\text{\normalfont npr}}    % negative pairs on the right 
\def\sr{\text{\normalfont sr}}        % singletons on the right
\def\ssr{\text{\normalfont ssr}}     % singletons on the strict right 
\def\out{{\rm out}}         % outer components 
\def\outsr{{\rm out \backslash sr}}     % outer components without singletons on the right 

\def\Pair{\text{\normalfont Pair}}      % pair blocks
\def\NPair{\text{\normalfont NPair}}    % negative pairs blocks 
\def\Sing{\text{\normalfont Sing}}     % singletons  

\def\id{I} 
\newcommand{\e}{{\epsilon}}
\renewcommand{\epsilon}{\varepsilon}
\def\simcr{\overset{\rm cr}{\sim}}

\newcommand{\be}{\begin{equation}}
\newcommand{\ee}{\end{equation}}

\begin{document}

\title{Noncommutative probability  of type D}

\author{Marek Bo\.zejko\footnote{Instytut Matematyczny, Polska Akademia Nauk, Ulica \'Sniadeckich 8, 00-956 Warszawa, Poland. Email: bozejko@gmail.com}, 
Wiktor Ejsmont\footnote{
 Department of Mathematics and Cybernetics, 
Wroc\l aw University of Economics, ul.\ Komandorska 118/12 53-345 Wroc\l aw, Poland. Email: wiktor.ejsmont@gmail.com} %
 and Takahiro Hasebe\footnote{Department of Mathematics, Hokkaido University, North 10 West 8, Kita-ku, Sapporo 060-0810 Japan. Email: thasebe@math.sci.hokudai.ac.jp}
}
\date{}
\maketitle

\begin{abstract}
We construct a deformed Fock space and a Brownian motion coming from Coxeter groups of type D. The construction is analogous to that of the $q$-Fock space (of type A) and the 
$(\alpha,q)$-Fock space (of type B).  
\end{abstract}
\section{Introduction}
Several deformations of boson, fermion and full Fock spaces and Brownian motion have been proposed so far. Bo\.zejko and Speicher used Coxeter groups of type A to construct a deformed Fock space and Brownian motion \cite{BS91,BS94}. Bo\.zejko and Speicher also considered 
general (mainly finite) Coxeter groups in \cite{BS94}. We followed this idea in \cite{BEH15} and constructed an $(\alpha,q)$-Brownian motion on an $(\alpha,q)$-Fock space using Coxeter groups of type B.  
The commutation relation satisfied by the creation and annihilation operators on the $(\alpha,q)$-Fock space reads 
\be\label{WickB}
b_{\alpha,q}(x)b_{\alpha,q}^\ast(y)- q b_{\alpha,q}^\ast(y)b_{\alpha,q}(x)= \langle  x,y \rangle \id + \alpha \langle x, \overline{y} \rangle\, q^{2 N}
\ee
where $\alpha,q\in [-1,1]$, $N$ is the number operator, $x,y$ are vectors in an underlying Hilbert space $H$ and $y\mapsto \overline{y}$ is a selfadjoint involution on $H$.  The orthogonal polynomials associated to the distribution of the $(\alpha,q)$-Gaussian operator $b_{\alpha,q}(x)+b_{\alpha,q}^*(x)$ are called $q$-Meixner-Pollaczek polynomials satisfying the recurrence relation
\begin{equation}\label{recursion0}
t P_n^{(\alpha, q)}(t) = P_{n+1}^{(\alpha, q)}(t) +[n]_q(1 + \alpha q^{n-1})P_{n-1}^{(\alpha, q)}(t), \qquad n=0,1,2,\dots
\end{equation}
where $P_{-1}^{(\alpha, q)}(t)=0,P_0^{(\alpha, q)}(t)=1$.
 When $\alpha=0$ then we get $q$-Hermite orthogonal polynomials and when $q=0$ then we get the orthogonal polynomials associated to a symmetric free Meixner distribution \cite{BB}.
The moments of the Brownian motion of type B are given by
\begin{equation}\label{mixed moment}
\begin{split}
&\langle\Omega, (b_{\alpha,q}(x_1)+b_{\alpha,q}^\ast(x_1))\cdots (b_{\alpha,q}(x_{2n})+b_{\alpha,q}^\ast(x_{2n}))\Omega\rangle_{\alpha,q}\\
&\qquad\qquad=
\sum_{(\pi,f)\in \P_{2}^B(2 n)} \alpha^{\np(\pi,f)} q^{\Cr(\pi)+2\cnp(\pi,f)} \prod_{\substack{\{i<j\} \in \pi \\ f(\{i<j\})=1} }\langle x_i, x_j\rangle \prod_{\substack{\{i<j\} \in \pi\\ f(\{i<j\})=-1}} \langle x_i,\overline{x}_j\rangle, 
\end{split}
\end{equation}
where $\P_2^B(2 n)$ is the set of pair partitions of type B, $\np(\pi,f)$ is the number of negative pairs, $\Cr(\pi)$ denotes the number of the crossings of $\pi$ and $\cnp(\pi,f)$ is the number of covered negative pairs of $\pi$. For further information the reader is referred to \cite{BEH15} and the references therein.

The goal of this paper is to introduce a new deformation of the full Fock space, creation and annihilation operators and Brownian motion in terms of Coxeter groups of type D. Our strategy is to replace the Coxeter groups of type A or B in the previous works \cite{BS91,BEH15} by Coxeter groups of type D. 
Given this background, it seems natural to try to extend this theory to Coxeter groups of type C. But it is known that Coxeter groups of type B and type C are isomorphic, and 
one can check that for Coxeter groups of type C we get the orthogonal polynomials \eqref{recursion0} (i.e.\ we get the same probability measure as in the type $B$). This also follows from the fact that  $\sum_{\sigma\in B(n)}q^{\ell(\sigma)}=\sum_{\sigma\in C(n)}q^{\ell(\sigma)}$ (see Carter \cite[Proposition 10.2.5]{C89}), where $\ell$ denotes the length function on the Coxeter groups.

The plan of the paper is following. First we define a $q$-Fock space, a creation operator $\D^*(x)$ and an annihilation operator $\D(x)$ of type D in Section \ref{sec2}. Next, natural properties of creation and annihilation operators, including norm estimates and the commutation relations, are presented in Section \ref{sec2.2}. The probability distribution of a Brownian motion of type D, $\G(x)=\D(x)+\D^\ast(x), x\in H$,  is studied in the Section \ref{sec3.1}. The associated orthogonal polynomials satisfy the recurrence relation  
\begin{equation}
\begin{split}
&t P_n^{( q)}(t) = P_{n+1}^{( q)}(t) +[n]_q(1 + q^{n-1})P_{n-1}^{( q)}(t), \qquad n=2,3,4,\dots, \\
&P_0^{( q)}(t) =1, P_1^{( q)}(t) =t, P_{2}^{( q)}(t) =t^2-1,  
\end{split}
\end{equation}
which is not known in the literature to the authors' knowledge. This polynomial $P_n^{(q)}(t)$ looks like $ P_n^{(1, q)}(t)$, but they are different since $P_2^{(\alpha,q)}(t)= t^2-(1+\alpha)$ from \eqref{recursion0}. This is because the first Jacobi parameter for $P_n^{(q)}(t)$ is different from that for $P_n^{(1,q)}(t)$. This difference of first Jacobi parameter comes from the fact that the Coxeter group of type D can be realized as a subgroup of the Coxeter group of type B with index 2. 
%where $P_{-1}^{(q)}(t)=0$ and $P_0^{(q)}(t)=1$ as convention. 
The main theorem is placed in Section \ref{sec3.4} where we show a Wick formula of type D
\begin{equation}
\langle\Omega, \G(x_1)\cdots \G(x_{2n})\Omega\rangle_{q}=
\displaystyle\sum_{(\pi,f) \in \PD_2(2 n)}q^{\Cr(\pi)+2 \cnp(\pi,f)}\prod_{\substack{\{i<j\} \in \pi \\ f(\{i<j\})=1} }\langle x_i, x_j\rangle \prod_{\substack{\{i<j\} \in \pi\\ f(\{i<j\})=-1}} \langle x_i,\overline{x}_j\rangle. 
\end{equation}
It is described by pair partitions of type D which are introduced and studied in Sections \ref{sec3.2} and \ref{sec3.3}. 
It turns out that our Wick formula of type D generalizes the $t$-transformed classical Wick formula for $t=2$ ($q=1$) as well as the free Wick formula ($q=0$). 
Using this formula we show that the vacuum vector is not tracial with respect to the von Neumann algebra generated by our Brownian motion $\G(x)$, where $x$ runs over a real Hilbert space.  

Our deformed Fock space (both of type B and type D) raises natural interesting questions. Studying the von Neumann algebra generated by the Brownian motions of type D is one direction. In the case of type A, the $q$-deformed von Neumann algebra shares many properties with the free group factor (e.g.\ factoriality \cite{Ric05}), and it is even isomorphic to the free group factor for sufficiently small $q$ \cite{GS14}. The von Neumann algebra of type B contains the so-called $t$-deformed von Neumann algebra, which is well investigated by Wysoczanski \cite{Wys06} and Ricard \cite{Ric06}.  For further information see references in \cite{BEH15}. In our von Neumann algebra of type D, the vacuum state is not a trace (Proposition \ref{NTrace}). The first basic question would be asking if it is cyclic separating. Factoriality of the von Neumann algebra is also a natural question. 

 Another possible future direction is to search for a connection between our pair partitions of type D and noncrossing partitions of type D, the latter of which has been developed in the literature, e.g.\ in \cite{R97,AR04}. Note that our pair partitions of type D are related to Coxeter groups of type D in a natural way (see Remark \ref{TD}). Yet another direction is classical probability. It is known that $q$-Brownian motion of type A has a classical Markov process realization \cite{BKS97}, and its probabilistic properties and extensions have been studied by several authors (see e.g.\ \cite{BW05,BW14}). It is natural to ask if such a realization also exists in the type B and type D cases. 
 
Combining \cite{BS91,BEH15} and the present paper, we have constructed deformations of a full Fock space by the natural three families of finite Coxeter groups, of type A,B and D (type C is isomorphic to type B as already mentioned). We wonder if a similar construction exists for infinite Coxeter groups, e.g.\ affine Coxeter groups. 

The above questions are also listed in the end of this paper as open problems, together with some other problems. 
 
%Thus several questions arise from our construction. 

\section{Type D deformation of full Fock space} \label{sec2}

\subsection{Full Fock space}
Let $H$ be a complex Hilbert space with inner product $\langle\cdot,\cdot\rangle$ linear on the right component and anti-linear on the left. Let $\F$ be the algebraic full Fock space over $\H$
\begin{equation}
\F:= \bigoplus_{n=0}^\infty H^{\otimes n}
\end{equation} 
with convention that $H^{\otimes 0}=\C\Omega$ is a one-dimensional normed space along a unit vector $\Omega$. Note that elements of $\F$ are finite linear combinations of the elements from $H^{\otimes n}, n\in \N\cup\{0\}$ and we do not take the completion. 
We equip $\F$ with the inner product  
\begin{align}
&\langle x_1 \otimes \cdots \otimes x_m, y_1 \otimes \cdots \otimes y_n\rangle_{0}= \delta_{m,n}\prod_{i=1}^n \langle x_i, y_i\rangle, \\   
&\langle \Omega, y_1 \otimes \cdots \otimes y_n\rangle_{0}= \langle x_1 \otimes \cdots \otimes x_m, \Omega\rangle_{0} =0, \\
&\langle \Omega, \Omega\rangle_{0}=1, 
\end{align} 
where $m, n \geq1$ and $x_i, y_i \in \H$. 

For $x\in H$ the free left creation and annihilation operators $\la^\ast(x), \la(x)$ on $\F$  are defined by \cite{V85}
\begin{align}
&\la^\ast(x)(x_1 \otimes \cdots \otimes x_n):= x \otimes x_1 \otimes \cdots \otimes x_n , &&n\geq 1\\
&\la^\ast(x)\Omega=x, \\
&\la(x)(x_1 \otimes \cdots \otimes x_n):=\langle x, x_1\rangle\, x_2 \otimes \cdots \otimes x_{n},&&n\geq2, \\ 
&\la(x)x_1:= \langle x,x_1 \rangle\,\Omega, \\ 
&\la(x)\Omega =0. 
\end{align}
It then holds that $\la^\ast(x)^\ast = \la(x)$ and 
$\la^\ast: H \to \mathbb{B}(\F)$ is linear, but  $\la: H \to \mathbb{B}(\F)$ is anti-linear.

\subsection{Coxeter groups of type D}\label{sec2.1}
Let $S(\pm n)$  be the group of all permutations of the $2n$ numbers $\pm1,\cdots, \pm n$. The Coxeter group of type B, denoted by $B(n)$, is the set of permutations $\sigma$ in $S(\pm n)$ such that 
$\sigma(-k)=-\sigma(k), k=1,\dots,n$. The Coxeter group $B(n)$ is a subgroup of $S(\pm n)$ generated by $\{\pi_1,\dots,\pi_{n-1},\overline{\pi}_n\}$, where $\overline{\pi}_n=(n,-n), \pi_i =(i,i+1)(-i, -i-1)$, $i=1,\dots,n-1$. Our strategy
is to define a Coxeter group of type $D(n)$ as a subgroup of the Coxeter group $B(n)$ of index 2. This is well studied in the literature (see \cite{BUR02}).  First we define $D(1)=\{e\}\subset B(1)=\grp\{\overline{\pi}_1\}$. Next we  define, as subgroup of $B(2)$ with index 2,  $D(2)=\grp\{\pi_1,\widehat{\pi}_1\}\subset B(2)$, where $\widehat{\pi}_1=\overline{\pi}_2\pi_1\overline{\pi}_2=(1,-2)$. Note that $D(2)$ is isomorphic to $\Z_2 \oplus \Z_2$. 
We define $D(3)$ by the Coxeter-Dynkin diagram 
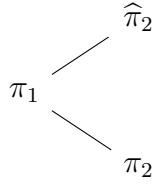
\begin{center}
\begin{tikzpicture}
  [scale=.5,auto=left,every node/.style={circle%,fill=blue!20
}]
  \node (n6) at (1,3) {$\pi_1$};
  \node (n4) at (4,5)  {$\widehat{\pi}_2$};
  \node (n3) at (4,1)  {$\pi_2$};
  \foreach \from/\to in {n6/n4,n6/n3}
    \draw (\from) -- (\to);
\end{tikzpicture}
\label{fig:DC2}
\end{center}
where $\widehat{\pi}_2=\overline{\pi}_3\pi_2\overline{\pi}_3$
so that the group $D(3)$ is isomorphic to the permutation group $S(4)$. 
We define $D(n)$ for $n\geq 4$ as a subgroup of $B(n)$ generated by $\pi_1,\dots, \pi_{n-1}, \widehat{\pi}_{n-1}$ where $\widehat{\pi}_{n-1}=\overline{\pi}_{n}\pi_{n-1}\overline{\pi}_{n}$. The Coxeter-Dynkin diagram for $D(n), n\geq4$ is described in Fig.\ \ref{fig:DCn}, 
 which says that the generators satisfy the generalized braid relations 
  $\pi_i^2=\widehat{\pi}_{n-1}^2=e, 1\leq i \leq n-1$,  $(\pi_{n-2}\widehat{\pi}_{n-1})^3=(\pi_i \pi_{i+1})^3=e, 1\leq i < n-1$ and $(\pi_i \pi_j)^2=(\pi_k \widehat{\pi}_{n-1})^2=e$ if $1\leq i,j,k\leq n-1,|i-j|\geq2,  k \neq n-2$.  Note that $\{\pi_i\mid i=1,\dots,n-1\}$ generates the symmetric group $S(n)$.

\begin{figure}[h]
\begin{center}
\begin{tikzpicture} 
  [scale=.5,auto=left,every node/.style={circle}]
  \node (n6) at (1,3) {$\pi_{n-2}$};
  \node (n4) at (4,5)  {$\widehat{\pi}_{n-1}$};
  \node (n5) at (-3,3)  {$\dots$};
\node (n1) at (-7,3)  {$\pi_{2}$};
\node (n2) at (-11,3)  {$\pi_{1}$};
  \node (n3) at (4,1)  {$\pi_{n-1}$};

  \foreach \from/\to in {n6/n4,n6/n3,n1/n2,n1/n5,n5/n6}
    \draw (\from) -- (\to);
\end{tikzpicture}
\caption{Coxeter-Dynkin diagram for $D(n).$}
\label{fig:DCn}
\end{center}
\end{figure}
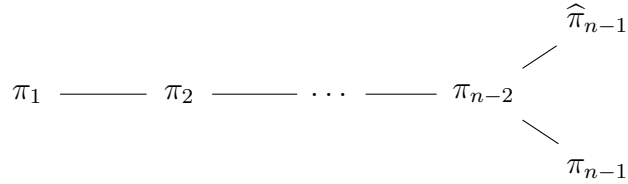

\subsection{Fock space of type D}
We deform the full Fock space $\F$ and creation $l^\ast(x)$ and annihilation operators $l(x)$ on it. 

A selfadjoint involution on $H$ is a selfadjoint linear bounded operator on $H$ such that the double application of it becomes the identity operator (see Example \ref{inv} below). 
Suppose that $x\mapsto \overline{x}, x\in H$ is a selfadjoint involution. We then define an action of $D(n)$ on $H^{\otimes n}$ by 
\begin{align}
&\pi_i(x_1\otimes\cdots \otimes x_n) = x_1 \otimes \cdots \otimes x_{i-1} \otimes x_{i+1} \otimes x_{i} \otimes x_{i+2}\otimes \cdots \otimes x_{n},& n \geq2,  \label{Coxeter1} \\
&\widehat{\pi}_{n-1}(x_1\otimes\cdots \otimes x_n)= x_1\otimes x_2\otimes\cdots \otimes x_{n-2} \otimes \overline{x}_{n}\otimes \overline{x}_{n-1},& n \geq 2. 
\label{Coxeter2}
\end{align}
\begin{remark}% (1).
The action of $\widehat{\pi}_{n-1}$ is defined in the above way since $\widehat{\pi}_{n-1} = \overline{\pi}_n \pi_{n-1} \overline{\pi}_n$ and the action of (the left version of) $\overline\pi_n$ defined in \cite{BEH15} is 
\be
\overline{\pi}_n(x_1\otimes\cdots \otimes x_n)= x_1\otimes x_2\otimes\cdots \otimes x_{n-1} \otimes \overline{x}_{n}, \qquad n \geq 1. 
\ee
A second way to define a subgroup of $B(2)$ with index 2 is $\grp\{\overline{\pi}_2\pi_1\overline{\pi}_2\pi_1\}$, but in this situation we have $\overline{\pi}_2\pi_1\overline{\pi}_2\pi_1(x_1\otimes  x_2)=\overline{x}_{1}\otimes \overline{x}_{2}$, which is not compatible with $D(n)$ for $n\geq 3.$ 
%\\(2). It is worth to mention about the article \cite{SS64}, where it is show
%that the involution is important  in the case $q=-1$.
\end{remark}
\begin{example} \label{inv}
\begin{enumerate}[\rm(1)] 
\item The identity involution $x\mapsto x$ on $H$. Then the actions of $\widehat{\pi}_{n-1}$ and $\pi_{n-1}$ coincide. 
\item If $H$ is spanned by an orthonormal basis $\{e_i\}_{i \in \{\pm 1,\dots, \pm n\}}$ (or $\{e_i\}_{i \in \mathbb{Z}\setminus\{0\}}$ in the infinite-dimensional case), then the map
$$
\overline{e}_{i}=e_{-i}, \qquad i \in \{\pm1,\dots,\pm n\}
$$
extends to an involution on $H$. 
\end{enumerate}
\end{example}

 Let $\ell$ be the length function on the Coxeter groups: $\ell(\sigma)$ is the minimal number $k$ such that $\sigma$ can be written as the product of $k$ generators, allowing multiple use of each generator.  
For $q \in[-1,1]$ we define the {\it symmetrizer of type D} on $H^{\otimes n}$, 
 \begin{align}
&\widehat{P}_{q}^{(n)}= \sum_{\sigma \in D(n)}  q^{\ell(\sigma)}\, \sigma,\qquad n \geq1, \\ 
&\widehat{P}_{q}^{(0)}= \id_{H^{\otimes 0}} 
\end{align}
with convention $0^0=1$. Our operator $\widehat{P}_{q}^{(n)}$ is a special case of the operator defined in \cite{BSz03}. Note that we have $\widehat{P}_{0}^{(n)}=\id_{\H^{\otimes n}}$. 
Moreover let 
$$
\widehat{P}_{q}=\bigoplus_{n=0}^\infty \widehat{P}_{q}^{(n)}
$$ 
be the symmetrizer of type D on the algebraic full Fock space $\F$. 
From Bo\.zejko and Speicher \cite[Theorem 2.2]{BS94}, the operator $\widehat{P}_{q}^{(n)}$ and hence $\widehat{P}_{q}$ is positive for $-1 \leq q \leq 1$. If $ |q|<1$ then $\widehat{P}_{q}^{(n)}$ is a strictly positive operator meaning that it is positive and $\text{Ker}(\widehat{P}_{q}^{(n)})=\{0\}$. 

The symmetrizer of type D allows us to define the deformation
\begin{equation}
\langle f, g\rangle_{q}:=\langle f, \widehat{P}_{q} g\rangle_{0}, \qquad f,g \in \F,  
\end{equation}
which is a semi-inner product from the positivity of $\widehat{P}_{q}$ for $q\in[-1,1]$. We restrict the parameters to the case $q \in(-1,1)$ so that the deformed semi-inner product becomes an inner product. 

\begin{definition} For $q\in(-1,1)$, the algebraic full Fock space $\F$ equipped with the inner product $\langle\cdot,\cdot \rangle_{q}$ is called the \emph{($q$-) Fock space of type D} and is denoted by $\Fq$. Let $\D^\ast(x):= \la^\ast(x)$ and $\D(x)$ be its adjoint in $\Fq$. The operators $\D^\ast(x)$ and $\D(x)$ are called \emph{($q$-) creation and annihilation operators of type D}, respectively.  
\end{definition}
More precisely, one can show that $\D^\ast(x)$ is a bounded operator from $H^{\otimes n}$ to $H^{\otimes (n+1)}$ for each $n \geq0$, and so 
$\D(x)\colon H^{\otimes (n+1)}\to H^{\otimes n}$ is defined to be its adjoint. They can then be extended to linear operators on $\Fq$ by direct sum. We see in Proposition \ref{norm}  that they are in fact bounded operators for $q\in(-1,1)$ on $\Fq$. 
One can check that $\D^\ast\colon H \to \mathbb{B}(\Fq)$ is linear and $\D\colon H \to \mathbb{B}(\Fq)$ is anti-linear, similarly to the free case $l^\ast, l$ on $\F$. 

Since $\widehat{P}_{0}$ is identity, our $q$-Fock space of type D is the full Fock space when $q=0$, and  $d_0^\ast(x)=l^\ast(x)$ and $d_0(x)=l(x)$.  
%Thus our type D setting actually generalizes the type A setting.  // I think this is not true. 

\begin{remark} 
In the type B case \cite{BEH15} we defined the symmetrizer with two parameters $\alpha,q \in[-1,1]$ by 
\be
P_{\alpha,q}^{(n)} = \sum_{\sigma \in B(n)} \alpha^{\ell_1(\sigma)} q^{\ell_2(\sigma)} \sigma, 
\ee 
where $\ell_1(\sigma)$ is the number of $\overline{\pi}_n$ that appear in an irreducible form of $\sigma$ and $\ell_2(\sigma) = \ell(\sigma)-\ell_1(\sigma)$. 
However, in the type D case the number of $\widehat{\pi}_{n-1}$ in an irreducible form of $\sigma \in D(n)$ is not well defined since $\pi_{n-2} \widehat{\pi}_{n-1} \pi_{n-2} = \widehat{\pi}_{n-1} \pi_{n-2} \widehat{\pi}_{n-1}$. Therefore, we introduce only one parameter $q$ in the type D case. 
\end{remark}

\section{Creation and annihilation operators of type D}\label{sec2.2}

\subsection{Recursive formula for symmetrizer of type D}
A natural embedding of $D(n-1)=\grp\langle \pi_1, \dots, \pi_{n-2},\widehat{\pi}_{n-2}\rangle$ into 
$D(n)=\grp\langle \pi_1, \dots, \pi_{n-1},\widehat{\pi}_{n-1}\rangle$ is defined by 
\begin{align}
&\pi_k \mapsto \pi_{k+1}, \qquad k=1,2,\dots, n-2, \\
& \widehat{\pi}_{n-2} \mapsto \widehat{\pi}_{n-1}. 
\end{align} 
Corresponding to the quotient $D(n-1) \backslash D(n)$, we  decompose the operator $\widehat{P}^{(n)}_{q}$. This decomposition is important throughout the paper. 
\begin{proposition}\label{prop1}
We have the decomposition 
\begin{equation}\label{decomposition}
\begin{split}
&\widehat{P}^{(n)}_{q}= ( I\otimes \widehat{P}^{(n-1)}_{q}) \widehat{R}^{(n)}_{q}=  (\widehat{R}^{(n)}_{q})^* ( I\otimes \widehat{P}^{(n-1)}_{q}), \qquad n\geq 2, \\
&\widehat{P}^{(1)}_{q} = \widehat{R}^{(1)}_{q},  
\end{split}
\end{equation}
where $\widehat{R}^{(n)}_{q}$ is a bounded linear operator on $\H^{\otimes n}$ defined by 
\begin{equation}
\widehat{R}^{(n)}_{q}= 
\begin{cases}\displaystyle  \id+\sum_{k=1}^{n-1}q^{k}\pi_{1}\cdots \pi_{k}   \\
\quad \displaystyle+ q^{n-1} \pi_1 \cdots \pi_{n-2} \widehat{\pi}_{n-1} \left(\id+\sum_{k=1}^{n-1}q^{k}\pi_{n-1}\cdots \pi_{n-k}\right), & n \geq3, \\
 \id + q \pi_1 + q \widehat{\pi}_1(\id+q  \pi_1), & n=2, \\
 \id, & n=1.  
\end{cases}
\end{equation}
\end{proposition}
\begin{proof} Let $n\geq2$. 
It is known that there exist unique left coset representatives $\{w(k)\mid 0 \leq k \leq 2n-1\}$ for $D(n-1)\backslash D(n)$ with minimal lengths \cite[p.\ 19]{Hum90}. 
Due to Stumbo \cite[Theorem 4]{S00}, these coset representatives are given by 
\be
w(k) = \left\{ \begin{array}{ll}
\pi_{1} \cdots \pi_{k} & \textrm{if $0\leq k\leq n-1$}, \\
\pi_{1}  \cdots  \pi_{n-2} \widehat{\pi}_{n-1}& \textrm{if $k=n$}, \\
\pi_{1}  \cdots  \pi_{n-2} \widehat{\pi}_{n-1}\pi_{n-1}\dots  \pi_{2n-k}& \textrm{if $ n+1\leq  k\leq  2n-1 $}. 
\end{array} \right.
\ee
Therefore every $\sigma \in D(n)$ decomposes into $\sigma=\sigma' w(k)$ for some unique $\sigma' \in D(n-1)$ and unique $k$, and in this case it is known that $\ell(\sigma)=\ell(\sigma') +\ell(w(k))$ (see \cite[p.\ 19]{Hum90}). Since $\widehat{R}_q^{(n)}$ is written in the form 
\be
\widehat{R}^{(n)}_{q} = \sum_{k=0}^{2n-1} q^{\ell(w(k))} w(k), 
\ee
we have the first identity \eqref{decomposition}. The second identity follows by the selfadjointness of $\widehat{P}_q^{(n)}$.  
\end{proof}

The operator $\widehat{R}_{q}^{(n)}$ plays a central role in this paper. Firstly we compute the annihilation operator in terms of $\widehat{R}_{q}^{(n)}$. For this purpose we need a commutation relation between $l^\ast(x)$ and $\pi_k$.

\begin{lemma} \label{lemm:relacjalpi}
For $x\in H$ and $n\geq2$ we have the following relations on $H^{\otimes n}$:  
\begin{align*}
&\la^\ast(x)\pi_k=\pi_{k+1}\la^\ast(x), \qquad 1 \leq k \leq n-1, \\
&\la^\ast(x)\widehat{\pi}_{n-1}=\widehat{\pi}_n\la^\ast(x). 
\end{align*}
This implies  $\la^\ast(x) \widehat{P}^{(n)}_{q} =( I\otimes \widehat{P}^{(n)}_{q}) \la^\ast(x)$ for $n\geq1$. 
\end{lemma}
\begin{proof}
Let $n \geq2$ and $f=x_1 \otimes  \cdots \otimes x_{n} \in H^{\otimes n}$. For $1\leq k<n $ we have 
\begin{align*}\la^\ast(x)\pi_k f=x\otimes x_1 \otimes \cdots  \otimes x_{k+1} \otimes x_{k} \otimes \cdots \otimes x_{n}=\pi_{k+1}\la^\ast(x)f
\end{align*}
and 
\begin{align*}\la^\ast(x)\widehat{\pi}_{n-1} f=x\otimes x_1 \otimes \cdots  \otimes \overline{x}_{n} \otimes \overline{x}_{n-1}=\widehat{\pi}_{n}\la^\ast(x)f. 
\end{align*}
\end{proof}

\subsection{Formula for annihilation operator}

Now we are ready to compute the annihilation operator in terms of $\widehat{R}_{q}^{(n)}$.

\begin{proposition}\label{prop2} For $n \geq1$, we have 
\begin{equation}
\D(x)=\la(x) {\widehat{R}^{(n)}}_{q} \text{~on $H^{\otimes n}$}. 
\end{equation}
\end{proposition}
\begin{proof}
Let $f \in H^{\otimes (n-1)},g \in H^{\otimes n}$. Recalling that $\la^\ast(x) \widehat{P}^{(n-1)}_{q} =( I\otimes \widehat{P}^{(n-1)}_{q}) \la^\ast(x)$ (see Lemma \ref{lemm:relacjalpi}) we get  
\begin{equation}
\begin{split}
\langle f, \D(x)g \rangle_{q} 
&=\langle \D^\ast(x)f, g \rangle_{q} =\langle \la^\ast(x) f, g \rangle_{q}=\langle \widehat{P}^{(n)}_{q} \la^\ast(x) f,  g \rangle_{0}  \\
&= \langle \la^\ast(x) f,    ( I\otimes \widehat{P}^{(n-1)}_{q}) \widehat{R}^{(n)}_{q} g \rangle_{0} =\langle( I\otimes \widehat{P}^{(n-1)}_{q}) \la^\ast(x) f,   \widehat{R}^{(n)}_{q} g \rangle_{0} \\
&= \langle\la^\ast(x) \widehat{P}^{(n-1)}_{q} f,  \widehat{R}^{(n)}_{q} g \rangle_{0}=\langle \widehat{P}^{(n-1)}_{q} f,   \la(x)\widehat{R}^{(n)}_{q} g \rangle_{0}\\
&= \langle f , \la(x)\widehat{R}^{(n)}_{q}g \rangle_{q}, 
\end{split}
\end{equation}
the conclusion. 
\end{proof}

We compute the annihilator $\D(x)$ more explicitly. 
Let $N$ be the number operator on $\Fq$ defined by
\be 
N(f)=n f, \qquad f\in H^{\otimes n}, n\in \N\cup\{0\}
\ee
 and let $J$ be the operator $0 \oplus \bigoplus_{n=1}^\infty \overline\pi_n$ on $\Fq$, that is, 
\begin{align}
J(x_1\otimes \cdots \otimes x_n) &=x_1\otimes \cdots \otimes x_{n-1} \otimes \overline{x}_n, \qquad n\geq1, \\
 J(\Omega)&=0. 
\end{align}
The operator $J$ is in fact a \emph{selfadjoint} involution on $\Fq$, which is proved in Corollary \ref{J}.

Let $\la_q$ and $\ra_q$ be left and right $q$-derivatives respectively: for $x, x_1,\dots, x_n \in \H, n \geq 1$, 
\begin{align}
&\la_q(x)(x_1\otimes \cdots \otimes x_n)= \sum_{k=1}^n q^{k-1} \langle x, x_k \rangle\, x_1\otimes \cdots \otimes \check{x}_{k} \otimes \cdots \otimes x_n, \label{rq}\\
&\ra_q(x)(x_1\otimes \cdots \otimes x_n)  = \sum_{k=1}^n  q^{k-1}  \langle x,x_{n-k+1}\rangle\, x_1\otimes\dots\otimes \check{x}_{n-k+1}\otimes \cdots \otimes  x_n, 
\label{lq1} \\
& \la_q(x)(\Omega)= \ra_q(x)(\Omega)=0,  
\end{align}
where $\check{x}_k$ means that $x_k$ is removed from the tensor, e.g.\ $x \otimes \check{y} \otimes z = x \otimes z$.

\begin{theorem}\label{thm1} For $x\in \H$ we have 
\be\label{eq:annihilator}
\D(x)= \la_q(x)+ q^N J \ra_{q}(\overline{x})= \la_q(x)+ J \ra_{q}(\overline{x})q^{N-1}.  
\ee
Note that $d_0(x)=\la_0(x)=\la(x)$ is the free left annihilation operator. 
\end{theorem}
\begin{proof} 
Let $n\geq2$. From Propositions \ref{prop1} and \ref{prop2} we have 
\begin{equation}
\begin{split}
\D(x)(x_1\otimes \cdots \otimes x_n)
&= \la(x)\widehat{R}_{q}^{(n)}(x_1\otimes \cdots  \otimes x_n) = L+R,  
\end{split}
\end{equation}
where 
\begin{align}
&L = \la(x)\left(I+\sum_{k=1}^{n-1}q^{k}\pi_{1}\cdots \pi_{k}\right)(x_1\otimes \cdots  \otimes x_n), \\
&R=  q^{n-1}  \la(x) \pi_1 \cdots \pi_{n-2} \widehat{\pi}_{n-1}\left(I+\sum_{k=1}^{n-1}q^{k}\pi_{n-1}\cdots \pi_{n-k}\right)(x_1\otimes \cdots  \otimes x_n).
\end{align}
After some computations, we get 
\begin{align}
&L=\la_q(x)(x_1\otimes \cdots  \otimes x_n), \\
&R= q^{n-1} \sum_{k=1}^n  q^{k-1}  \langle x,\overline{x}_{n-k+1}\rangle\, J(x_1\otimes\dots\otimes \check{x}_{n-k+1}\otimes \cdots \otimes  x_n).
\label{lq}
\end{align}
By the selfadjointness of the involution we have $\langle x,\overline{x}_{n-k+1}\rangle=\langle \overline{x},x_{n-k+1}\rangle$, and hence we get the conclusion \eqref{eq:annihilator} on $\H^{\otimes n}$  for $n\geq2$. 
For $n=0,1$ we can directly check the formula. 
\end{proof}

\subsection{Commutation relations and norm estimate}   

The commutation relations on the one particle space $H^{\otimes 1}$ and on the other spaces $H^{\otimes n}, n\neq1$ look different.  
\begin{proposition}\label{commutation}
For $x,y \in H$ we have the commutation relation 
\begin{align}
\D(x)\D^\ast(y)- q \D^\ast(y)\D(x)&= \langle  x,y \rangle \id +  \langle x, \overline{y} \rangle\, q^{2 n}J \textrm{~~ on~ } \H^{\otimes n}, n=0,2,3,4,\dots  \label{comutacja}
\end{align}
\end{proposition}
\begin{remark}
For $n=1$,  the commutation relation has the form
\begin{align}
 \D(x)\D^\ast(y)z- q \D^\ast(y)\D(x)z&= \langle  x,y \rangle z + q \langle  x,\overline{z} \rangle \overline{y} + q^2 \langle x, \overline{y} \rangle\, \overline{z} \textrm{ ~~for~  } z \in H. \label{comdlan1}
\end{align}
The commutation relation for $n\neq1$ looks similar to the type B case \cite{BEH15}
$$
b_{\alpha,q}(x)b_{\alpha,q}^\ast(y)- q b_{\alpha,q}^\ast(y)b_{\alpha,q}(x)= \langle  x,y \rangle \id + \alpha \langle x, \overline{y} \rangle\, q^{2 N},  
$$
while $n=1$ case appears quite different. 
\end{remark}
\begin{proof}
Let $n\geq2$. From Theorem \ref{thm1}, it holds that 
\begin{align*}
\D(x)\D^\ast(y)(x_1\otimes \cdots \otimes x_n)
= & \sum_{k=1}^n q^{k} \langle x,x_k \rangle \,y \otimes  x_1\otimes \cdots \otimes \check{x}_k \otimes \cdots \otimes x_n \\
&+ \sum_{k=1}^n q^{n+k-1} \langle x,\overline{x}_{n-k+1} \rangle \,  J(y \otimes x_1\otimes \cdots \otimes \check{x}_{n-k+1} \otimes \cdots \otimes x_n)\\
&+\langle  x,y \rangle\,x_1\otimes \cdots \otimes x_n  +  q^{2 n}\langle x,\overline{y }\rangle\, x_1\otimes \cdots \otimes \overline{x}_n 
\end{align*}
and 
\begin{align*}
q\D^\ast(y)\D(x)(x_1\otimes \cdots \otimes x_n)
=&  \sum_{k=1}^n   q^{k}  \langle x, x_k \rangle \,y \otimes  x_1\otimes \cdots \otimes \check{x}_k\cdots \otimes x_n \\
&+ \sum_{k=1}^n q^{n+k-1} \langle x,\overline{x}_{n-k+1} \rangle \,  J(y \otimes x_1\otimes \cdots \otimes \check{x}_{n-k+1} \otimes \cdots \otimes x_n), 
\end{align*}
and the conclusion follows. %For $n=0,1$ the proof follows from direct calculation. 
\end{proof}

\begin{corollary}\label{J}
The operator $J$ is a selfadjoint involution on (the completion of) $\Fq$. In particular, $\|J\|_q=1$. 
\end{corollary}
\begin{proof}
The operator $J$ is an involution since so is $\overline{\pi}_n$. When $q=0$, the Fock space of type D is the full Fock space and the selfadjointness of $\overline{\pi}_n$ implies the selfadjointness of $J$. When $q\neq0$, take $x\in H$ such that  $\|x\|=1, \overline{x}=\pm x$. The commutation relation \eqref{comutacja} for $y=x$ reads 
\begin{equation}
J= \pm q^{-2n}  (\D(x)\D^\ast(x) -q \D^\ast(x) \D(x) -1)~~\text{on~}H^{\otimes n}, n=0,2,3,\dots, 
\end{equation} 
so $J$ is selfadjoint on $(H^{\otimes n}, \|\cdot\|_q), n\neq1$. Finally,  since $J|_{H^{\otimes1}}=\pi_1|_{H^{\otimes1}}$, the restriction of $J$ to $H^{\otimes1}$ is selfadjoint. Note that $\widehat{P}^{(1)}_q=\id$ and so the inner product is not deformed on $H^{\otimes 1}$. 

Any selfadjoint involution $A$ on a Hilbert space has the operator norm 1  since $\|A\|^2= \|A^\ast A\| = \|A^2\| = \|A\|$. 
\end{proof}

We study the norm of the creation operator of type D. 
Let $[n]_q$ be the $q$-number 
\be
[n]_q:= 1+q+\cdots+q^{n-1},\qquad n \geq1. 
\ee
% and let $[n]_q!$ be the $q$-factorial 
%\be
%[n]_q !:= [1]_q \cdots [n]_q,\qquad n \geq1. 
%\ee

\begin{proposition} \label{norm} Suppose that $x \in H$. 
\begin{enumerate}[\rm(1)]
\item\label{Norm1} For $-1 < q < 1$, we have $
\frac{\|x\|}{\sqrt{1-q}}\leq \|\D^\ast(x)\|_{q}; 
$
\item\label{Norm2} For $0 \leq q < 1$, we have $ \|\D^\ast(x)\|_{q}\leq \sqrt{\frac{2 }{1-q}}\|x\|;
$
\item\label{Norm3} For $-1 < q < 0$, we have $ \|\D^\ast(x)\|_{q}\leq \sqrt{ 1+|q| +q^{2}} \|x\|. 
$
\end{enumerate}  

\end{proposition}
\begin{proof}
\eqref{Norm1} Lower bound. For $n\geq2$ it follows that
\begin{align*} 
\|x^{\otimes n}\|_{q}^2 &= \langle x^{\otimes n},  \widehat{P}_{q}^{(n)} x^{\otimes n}\rangle_{0} \\
&=  \langle( I\otimes \widehat{P}^{(n-1)}_{q}) x^{\otimes n}, {\widehat{R}^{(n)}}_{q} x^{\otimes n}\rangle_{0} \\
&= [n]_q \langle ( I\otimes \widehat{P}^{(n-1)}_{q}) x^{\otimes n},   x^{\otimes n}\rangle_{0} + q^{n-1}[n]_q \langle ( I\otimes \widehat{P}^{(n-1)}_{q}) x^{\otimes n},  \overline{x} \otimes x^{\otimes (n-2)} \otimes \overline{x}\rangle_{0}\\
&=[n]_q \|x\|^2\langle \widehat{P}_{q}^{(n-1)} x^{\otimes n-1}, x^{\otimes (n-1)}\rangle_{0}  + \langle x,\overline{x}\rangle\, q^{n-1}[n]_q\langle \widehat{P}_{q}^{(n-1)} x^{\otimes (n-1)}, x^{\otimes (n-2)} \otimes \overline{x} \rangle_{0} 
\\
&=[n]_q \|x\|^2\|x^{\otimes (n-1)}\|_{q}^2  + \langle x,\overline{x}\rangle\, q^{n-1}[n]_q \langle  x^{\otimes (n-1)}, x^{\otimes (n-2)} \otimes \overline{x} \rangle_{q}, 
\end{align*}
and so 
\begin{equation}\label{eq0000}
\begin{split}
\|\D^\ast(x)x^{\otimes (n-1)}\|_{q}^2
&=\|x^{\otimes n}\|_{q}^2 \\
&= [n]_q \|x\|^2\|x^{\otimes (n-1)}\|_{q}^2  + \langle x,\overline{x}\rangle\, q^{n-1}[n]_q \langle  x^{\otimes (n-1)}, x^{\otimes (n-2)} \otimes \overline{x} \rangle_{q} . 
\end{split}
\end{equation}
The limit $n\to\infty$ gives the lower bound. 
\newline
\newline
\eqref{Norm2} Upper bound for $0 \leq q < 1$.  The proof follows the line of \cite[Lemma 4]{BS91}. As already noted, \cite[Theorem 2.2]{BS94} guarantees that the operator $\widehat{P}_q^{(n)}$ is positive, so
\begin{equation}
\begin{split}
(\widehat{P}_{q}^{(n)})^2 
&=  (\widehat{P}_{q}^{(n)})^\ast \widehat{P}_{q}^{(n)}  \\
&= ( I\otimes \widehat{P}_{q}^{(n-1)})^\ast  (\widehat{R}_{q}^{(n)}) (\widehat{R}_{q}^{(n)})^\ast (I\otimes  \widehat{P}_{q}^{(n-1)} )   \\
&\leq  \|\widehat{R}_{q}^{(n)}\|_{0}^2 (I \otimes \widehat{P}_{q}^{(n-1)} )^\ast(I\otimes  \widehat{P}_{q}^{(n-1)})  \\
&= \|\widehat{R}_{q}^{(n)}\|_{0}^2 ( I\otimes (\widehat{P}_{q}^{(n-1)})^2).  
\end{split}
\end{equation}  Since $\|\pi_i\|_{0} =1$ and $\|\widehat{\pi}_i\|_{0} =1$ for all $i$,  we have $\|{\widehat{R}^{(n)}}_{q}\|_{0} \leq (1+ q^{n-1})[n]_q,$ for $n \geq 1$. 
By taking the square root of operators  one gets the inequality 
\begin{equation}\label{eq0002}
\begin{split}
\widehat{P}_{q}^{(n)} &\leq  \|\widehat{R}_{q}^{(n)}\|_{0} (I\otimes  \widehat{P}_{q}^{(n-1)})  \\
&\leq (1+q^{n-1})[n]_q ( I\otimes\widehat{P}_{q}^{(n-1)}) \\
&\leq \frac{2}{1-q}I\otimes \widehat{P}_{q}^{(n-1)}. 
\end{split}
\end{equation}
 Therefore we have for $f\in H^{\otimes n}$ that 
\begin{equation}\label{eq0003}
\begin{split}
\langle \D^\ast(x) f,  \D^\ast(x) f \rangle_{q} 
&= \langle \widehat{P}_{q}^{(n+1)}(x\otimes  f),  x\otimes  f\rangle_{0} \\
&\leq \frac{2}{1-q} \langle x \otimes (\widehat{P}_{q}^{(n)}f) ,x\otimes  f \rangle_{0} \\
&= \frac{2}{1-q} \langle \widehat{P}_{q}^{(n)}f,  f\rangle_{0} \langle x, x\rangle \\
&=\frac{2}{1-q} \|f\|_{q}^2 \|x\|^2.   
\end{split}
\end{equation}
This inequality holds true for $f$ in $\Fq$ which is the direct sum of $H^{\otimes n}$.  Thus we obtain $\|\D^\ast(x) f\| \leq  \sqrt{\frac{2}{1-q}} \|f\|_{q} \|x\|$ for any $f\in \Fq$. 
\newline
\newline
\eqref{Norm3}  Upper bound for $-1 < q < 0$.  The operators $\D(x)\D^\ast(x)$ and $-q\D^\ast(x)\D(x)$ are positive and map $H^{\otimes n}$ into itself for $n\geq0$. So for $ \xi\in H^{\otimes n}, \|\xi\|_q=1, n\geq 2$, 
\begin{align*}
\|\D^\ast(x)\xi \|_q^2
&= \langle \xi, \D(x)\D^\ast(x)\xi\rangle_q \leq \langle\xi, \D(x)\D^\ast(x)\xi-q\D^\ast(x)\D(x)\xi\rangle_q
\\&\leq \|x\|^2+|\langle  x,\overline{x} \rangle q^{2 n}| \|J\|_q \leq (1+q^2)\|x\|^2,   
\end{align*}  
where we used the commutation relation \eqref{comutacja} and $\|J\|_q=1, \|\overline{x}\|=\|x\|$ on the second line.   

For $\xi\in H, \|\xi\|_q = \|\xi\|=1$, a similar estimate shows that  
\begin{align*}
\|\D^\ast(x)\xi\|_q^2& \leq \langle\xi, \D(x)\D^\ast(x)\xi-q\D^\ast(x)\D(x)\xi\rangle_q \\
&\leq  \|x\|^2+|q|\|x\|^2 +q^{2}\|x\|^2, 				
\end{align*}
where we used equation \eqref{comdlan1}. Finally for $\Omega \in H^{\otimes 0}$ we have $\|\D^\ast(x)\Omega\|_q^2=\|x\|^2$. These estimates prove the conclusion. 
\end{proof}

\begin{remark} 
For $q=1$, our operators $\D(x)$ are unbounded.
\end{remark}

\section{Gaussian operator of type D}
\subsection{Orthogonal polynomials}\label{sec3.1}
\begin{definition}
The bounded selfadjoint operator 
\begin{equation}
\G(x)= \D(x) +\D^\ast(x),\qquad x \in H
\end{equation}
on $\Fq$ is called the \emph{($q$-) Gaussian operator of type D}. The family $\{\G(x)\mid x \in H\}$ is called the ($q$-) Brownian motion of type D. 
\end{definition}
In this section we study the probability distribution of the Gaussian operator of type D with respect to the vacuum state.  

Let $\mu$  be a probability measure on $\R$ with finite moments of all orders. The Gram-Schmidt method applied to the sequence $1,t,t^2,t^3,\dots$ in the Hilbert space $L^2(\R,\mu)$ yields a sequence of orthogonal polynomials $P_0(t),$ $P_1(t),$ $P_2(t), \dots$ with $\text{deg}\, P_n(t) =n$.  We take the normalization of $P_n(t)$ such that it becomes monic, i.e., the coefficient of $t^n$ is 1. It is known that they satisfy a recurrence relation
\begin{equation}\label{rec}
t P_n(t) = P_{n+1}(t) +\beta_n P_n(t) + \gamma_{n-1} P_{n-1}(t),\qquad n =0,1,2,\dots
\end{equation}
with the convention that $P_{-1}(t)=0$. The coefficients $\beta_n$ and $\gamma_n$ are called \emph{Jacobi parameters} and they satisfy $\beta_n \in \R,\gamma_n \geq 0$ and 
\begin{equation}\label{eq54}
\gamma_0 \cdots \gamma_{n-1}=\int_{\R}|P_{n}(t)|^2\mu(\d t),\qquad n \geq 1.
\end{equation} 
Conversely, given a sequence $\beta_n \in \R$ and $\gamma_n \geq0,$ $n=0,1,2,\dots,$ there exists a probability measure $\mu$ with finite moments such that it associates orthogonal polynomials determined by the recursion \eqref{rec}. 
More details are found in \cite{HO07}. 

Let $(P_n^{(q)}(t))_{n=0}^\infty$ be polynomials determined by the recursion relation 
\begin{equation}\label{recursion}
\begin{split}
&t P_n^{( q)}(t) = P_{n+1}^{( q)}(t) +[n]_q(1 + q^{n-1})P_{n-1}^{( q)}(t), \qquad n= 2,3,4,\dots, \\
& P_0^{(q)}(t)=1, P_1^{( q)}(t) =t,  P_{2}^{( q)}(t)=t^2-1.  
\end{split}
\end{equation}
This recursion means that $\beta_n=0$ for $n\geq0$ and $\gamma_0=1, \gamma_{n-1}=[n]_q(1 + q^{n-1})$ for $ n\geq2$.  
There exists a probability measure $\mu_q$ which associates the orthogonal polynomials $P_n^{(q)}(t)$. The distribution of the $q$-Gaussian operator of type D coincides with $\mu_q$. 

\begin{theorem} Suppose that $q\in(-1,1)$ and $x \in H, \|x\|=1, \overline{x}=\pm x$. 
The probability distribution of $\G(x)$ with respect to the vacuum state is given by $\mu_{q}$. 
\end{theorem}
\begin{remark}
We do not know what is the distribution of $\G(x)$ when $x$ is not an eigenvector of the involution $~\bar{}~$.  
\end{remark}
\begin{proof} %Under the assumption $\|x\|=1, \overline{x}=\pm x$, the formula \eqref{eq0000} implies that $\|x^{\otimes n}\|_q^2 = \gamma_{n-1} \|x^{\otimes (n-1)}\|_q^2$ for $n\geq2$. A direct calculation also shows that $\|x\|_q^2=\gamma_0\|\Omega\|_q^2=1$. This shows that 
%\begin{equation}\label{isom}
%\|x^{\otimes n}\|_q^2 = \gamma_0 \gamma_1 \cdots \gamma_{n-1} = \int_{\R}|P_n(t)|^2\mu_q(\d t),\qquad n \geq 0.
%\end{equation}
%Therefore the map $\Phi\colon (\overline{\text{span}}\{x^{\otimes n}\mid n \geq 0\}, \|\cdot\|_q) \to L^2(\R,\mu_q)$ defined by 
%\begin{equation}\label{unitary}
%\Phi(x^{\otimes n})= P_n^{(q)}(t),\qquad n \geq0
%\end{equation}
% is unitary. 

Using the formula $\D(x)= l(x)\widehat{R}_q^{(n)}$ on $\H^{\otimes n}$, we have 
\be
\begin{split}
&\G(x)\Omega = x,  \\
&\G(x)x = x^{\otimes 2} + \Omega
\end{split}
\ee
and, using $\overline{x}=\pm x$, 
\be\label{Ind}
\begin{split}
\G(x) x^{\otimes n} 
&=x^{\otimes (n+1)} + [n]_q  x^{\otimes (n-1)} + q^{n-1}[n]_q \langle x, \overline{x}\rangle\, x^{\otimes (n-2)} \otimes \overline{x} \\
& = x^{\otimes (n+1)} + [n]_q(1+q^{n-1})  x^{\otimes (n-1)}, \qquad n \geq  2. 
\end{split}
\ee
These formulas are of the same form as \eqref{recursion}. 

Let $P^{(q)}_n(t)$ be the orthogonal polynomials as above.  
We next show by induction that 
\be
P_n^{(q)}(\G(x))\Omega=x^{\otimes n}, \qquad n \geq0. 
\ee
This is true for $n=1,2$ since $P_{1}^{(q)}(\G(x))\Omega =\G(x)\Omega=x$ and $P_{2}^{(q)}(\G(x))\Omega =x^{\otimes 2}$. Using the induction hypothesis for $n-1,n$, \eqref{recursion} and \eqref{Ind} shows that, for $n\geq2,$ 
\begin{align*}
&P_{ n+1}^{(q)}( \G(x))\Omega =\G(x) P_{ n}^{(q)}( \G(x))\Omega-[n]_q (1+  q^{n-1}) P_{ n-1}^{(q)}( \G(x))\Omega =x^{\otimes(n+1)}.
\end{align*}

%Hence, from above it follows that
%\begin{equation}
%\|x^{\otimes n}\|_{q}= \|P_n^{(q)}\|_{L^2},\qquad n\in\N\cup\{0\}. 
%\end{equation}
%Therefore the map $\Phi\colon (\text{span}\{x^{\otimes n}\mid n \geq 0\}, \|\cdot\|_{q}) \to L^2(\R,\mu_{q})$ defined by $\Phi(x^{\otimes n})= P_n^{(q)}(t)$ is an isometry. 
 
%Since $\Phi$ is an isometry we get 
%$\langle \Omega,\G(x)^n(x)\Omega\rangle_{q} = m_n(\mu_{q})$ .
%In fact an induction argument shows that $\Phi(\G(x)^n \Omega) = t^n$. For even integers $n$, the unitarity of $\Phi$ shows that 

Since $P_n^{(q)}(t)$ and $P_0^{(q)}(t)=1$ are orthogonal in $L^2(\R,\mu_q)$ for $n\geq1$, we get 
\begin{equation}\label{Mom}
\langle\Omega, P_n^{(q)}(\G(x)) \Omega \rangle_q = \delta_{0,n} = \int_\R P_n^{(q)}(t)\, \d \mu_q(t). 
\end{equation}
Writing $t^n$ as the linear combination of $P_0^{(q)}(t), \dots, P_n^{(q)}(t)$, we obtain by induction on $n$ that 
\be\label{Moment}
\langle\Omega, \G(x)^n \Omega \rangle_q = \int_\R t^n\, \d \mu_q(t). 
\ee 
%which is the $n^{\rm th}$ moment of $\mu_q$. For odd integers $n$ we can show that $\langle\Omega, \G(x)^n \Omega \rangle_q = 0 = \int_\R t^n \,\mu_q(\d t).$  
Since the operator $\G(x)$ is bounded, its vacuum distribution is compactly supported. Hence the moment problem is determinate and we conclude that $\G(x)$ has the distribution $\mu_q$. 
\end{proof}

\begin{remark}  \begin{enumerate}[\rm(1)]
\item
The {\it $t$-transformation} (or Boolean convolution power by $t$) of a probability measure \cite{BW01} changes only the first Jacobi parameters $(\beta_0, \gamma_0)$ into $(t \beta_0, t \gamma_0)$ and keep the other Jacobi parameters unchanged. This shows that $\mu_q$ is the $1/2$-transformation of the Gaussian distribution of type B when $\alpha=1$, cf.\ \eqref{recursion0}.

 \item Expanding $P_m^{(q)}(t) P_n^{(q)}(t)$ as the linear combination of $t^k, k=0,1,\dots, m+n$ and using \eqref{Moment} show that 
\be
\langle \Omega, P_m^{(q)}(\G(x))P_n^{(q)}(\G(x)) \Omega \rangle_q =\int_\R P_m^{(q)}(t)P_n^{(q)}(t)\, \d \mu_q(t), 
\ee
which generalizes \eqref{Mom}. The left hand side is equal to $\langle P_m^{(q)}(\G(x))\Omega, P_n^{(q)}(\G(x)) \Omega \rangle_q = \langle x^{\otimes m}, x^{\otimes n}\rangle_q$ by selfadjointness of $P_n^{(q)}(\G(x))$. In particular, we obtain from \eqref{eq54} that 
\be\label{isom}
\|x^{\otimes n}\|_q^2= \gamma_0 \gamma_1 \cdots \gamma_{n-1}, 
\ee 
where $\gamma_0=1$ and $\gamma_{n-1}=[n]_q(1 + q^{n-1})$ for $ n\geq2$. 
 Since $\|x^{\otimes n}\|_q^2 = \langle x^{\otimes n}, \widehat P_q^{(n)} (x^{\otimes n}) \rangle_0$, the equation \eqref{isom} is equivalent to 
\begin{equation}
\sum_{\sigma\in D(n)}q^{\ell(\sigma)} = \gamma_0 \gamma_1 \cdots \gamma_{n-1}, \qquad n \geq1, 
\end{equation}  
which equals $[2]_q[4]_q\dots[2n-2]_q[n]_q$. This is exactly a formula in the book of Carter \cite[Theorem 10.2.3 and Proposition 10.2.5]{C89} and confirms our good choice of $D(2)$. 
\item In the limit $q\to1$ we get the $t$-transformed normal distribution $N(0,2)$ with $t=1/2$. 

\item If $q=0$ then we get the standard semicircle law $(1/2\pi)\sqrt{4-t^2}1_{(-2,2)}(t)\,\d t$. 
The orthogonal polynomials $P_n^{(0)}(t)$ are Chebyshev polynomials of the second kind. 
\end{enumerate}
\end{remark}

\subsection{Set partitions and partition statistics}\label{sec3.2}
From now on we study Wick's formula, that is, a formula for correlation functions of Brownian motion, or more generally, creation and annihilation operators. Combinatorics of set partitions is needed to describe that formula. 

Let $[n]$ be the set $\{1,\dots,n\}$. 
A \emph{pair} (or a pair block) $V$ of a set partition is a block with cardinality 2 and a \emph{singleton} of a set partition is a block with cardinality 1. The set of partitions of $[n]$ whose blocks have cardinality $1$ or $2$ is denoted by $\P_{1,2}(n)$. 
When $n$ is even, a set partition of $[n]$ is called a \emph{pair partition} if every block is a pair. The set of pair partitions of $[n]$ is denoted by $\P_2(n)$.

For subsets $A,B$ of $[n]$, we say that {\it $A$ is on the left of $B$} or {\it $B$ is on the right of $A$} if $\min A <  \min B$. We say that {\it $A$ is on the strict left of $B$} or {\it $B$ is on the strict right of $A$} if $\max A <  \min B$. 

For two subsets $A, B$ of $[n]$, we say that $A$ \emph{covers} $B$ if there are $i,j \in A$ such that $i <k <j$ for any $k\in B$. 

\begin{figure}[h]
\centering
\begin{minipage}{0.45\hsize}
\begin{center}
 \begin{tikzpicture}[scale = 1.2]
 \draw (0,0) -- (0,1) -- (2,1) -- (2,0); 
  \draw (1,0) -- (1,0.5) -- (3,0.5) -- (3,0); 
\end{tikzpicture}
\caption{$\{2,4\}$ is on the right of $\{1,3\}$, but not in the strict sense.}
\end{center}
\end{minipage}
\hspace{5mm}
\begin{minipage}{0.45\hsize}
\begin{center}
 \begin{tikzpicture}[scale = 1.2]
 \draw (0,0) -- (0,1) -- (1,1) -- (1,0); 
  \draw (2,0) -- (2,1) -- (3,1) -- (3,0); 
\end{tikzpicture}
\caption{$\{3,4\}$ is on the strict right of $\{1,2\}$.}
\end{center}
\end{minipage}

\end{figure}

We introduce several statistics of set partitions $\pi$.  
Let $\Pair(\pi)$ be the set of pair blocks and $\Sing(\pi)$ be the set of singletons of $\pi$. When writing pairs we sometimes simplify the notation into $\{i<j\} \in \Pair(\pi)$ instead of $\{i,j\} \in \Pair(\pi), i<j$. 

Let $\cover(V)$ be the number of blocks of $\pi$ which cover  $V$: 
$$
\cover(V)=\#\{W \in \pi  \mid  W \text{~covers~}V  \}. 
$$
Let $\cs(\pi)$ be the number of covered singletons, counting multiplicity of covers:    
$$
\cs(\pi)=\#\{(S,W) \in \Sing(\pi) \times \pi \mid W \text{~covers~}S  \}. 
$$
The number of singletons on the right of $V \in \pi$ is denoted by $\sr(V)$,  
$$
\sr(V)=\#\{S \in \Sing(\pi)\mid \text{$S$ is on the right of $V$}\}.  
$$
Similarly, the number of singletons on the strict right of $V \in \pi$ is denoted by $\ssr(V)$,  
$$
\ssr(V)=\#\{S \in \Sing(\pi)\mid \text{$S$ is on the strict right of $V$}\}. 
$$
The number of left crossings of $V \in \pi$ is defined by 
$$
\lcr(V)=\#\{W \in \pi \mid \exists i,j \in V, \exists k,l\in W, k<i<l<j\}. 
$$
Let $\Cr(\pi)$ be the number of all crossings of $\pi$ defined by 
\begin{equation*}
\begin{split}
\Cr(\pi) = \sum_{V\in \pi} \lcr(V).  
\end{split}
\end{equation*}

We then define connected components, in particular {\it outer connected components} of a partition $\pi$ of $[n]$.  Given two blocks $V,W$ of a partition we write $V \simcr W$ when $V$ and $W$ cross.  Then we write $V \sim W$ when $V=W$ or there exist blocks $V_0=V, V_1, \dots, V_{k-1},V_k=W$ of $\pi$ such that $V_i \simcr V_{i+1}$ for all $i=0,1,\dots, k-1$. 
The equivalence relation $\sim$  splits the partition $\pi$, regarded as a set, into equivalence classes $\pi_1, \dots, \pi_m$. 
%The equivalence relation $\sim$ decomposes the blocks of $\pi$ into equivalence classes $\pi_1, \dots, \pi_m$.
 Then $C_k:= \bigcup_{V\in \pi_k} V$ $(k=1,\dots, m)$ is called a {\it connected component} of $\pi$. 
By definition, two blocks which cross each other are contained in the same connected component. 
For example, the partition $\pi=\{\{1, 3, 7\}, \{2, 8\}, \{4, 5, 6\}\}$ has two connected components $\{1, 2, 3, 7, 8\}$ and $\{4, 5, 6\}$. 

Given two connected components, either one is on the strict right of the other, or one covers the other.  
A connected component is said to be {\it outer} if it is not covered by any other connected components. 
In the partition $\pi = \{\{1, 3, 7\}, \{2, 8\}, \{4, 5, 6\}\}$ the connected component $\{1, 2, 3, 7, 8\}$ is outer. 

We want to exclude singletons: the number of outer connected components which are not singletons is denoted by $\out(\pi)$. Moreover, $\outsr(\pi)$ denotes the number of connected components which are not singletons and which do not have singletons of $\pi$ on the right. In the partition $\pi = \{\{1, 3, 6\}, \{2, 5\}, \{4\}, \{7,8\}\}$ the outer connected component $\{1, 2, 3, 5, 6\}$ has a singleton $\{4\}$ on the right, so it is not counted into $\outsr(\pi)$. The other outer connected component $\{7,8\}$ is counted, so $\outsr(\pi)=1$.

\subsection{Set partitions of type D and colored partition statistics}\label{sec3.3}

A pair $(\pi,f)$ is called a \emph{colored set partition} if $\pi$ is a set partition and $f:\pi\to \{\pm1\}$ is a map, which means a coloring of the blocks of $\pi$ \cite{BEH15}. A block colored by $-1$ is called a {\it negative block} and a block colored by $1$ is called a {\it positive block}. 

Sometimes we want to think of partitions which are partially colored. When at least the pairs of $\pi$ are colored, namely a map $f\colon \Pair(\pi)\to\{\pm1\}$ is given, we denote the set of negative pairs by $\NPair(\pi,f)$,   
$$
\NPair(\pi,f)=\{W \in \Pair(\pi): f(W)=-1\}, 
$$
and by $\np(\pi,f)$ its cardinality, 
$$
\np(\pi,f) = \# \NPair(\pi,f). 
$$
 For simplicity we often drop the dependency on the coloring $f$ from the notation, like $\np(\pi)$.

The number of negative pairs on the right of $V \in \pi$ is denoted by $\npr(V)$: 
$$
\npr(V)=\#\{W \in \NPair(\pi): \text{$W$ is on the right of $V$}\}. 
$$
Let $\cnp(\pi,f)$ be the number of covered negative pairs counting multiplicity of covers: 
$$
\cnp(\pi,f)=\#\{(V,W) \in \NPair(\pi) \times \Pair(\pi) \mid W \text{~covers~}V  \}. 
$$
Let $\npssr(\pi,f)$ be the number of pairs
formed by a negative pair and a singleton on the strict right:  
\begin{align*}
\npssr(\pi,f)&=\#\{(V,S) \in \NPair(\pi) \times \Sing(\pi) \mid  \text{$S$ is on the strict right of $V$} \}. 
\end{align*}

\begin{definition} Suppose that $\pi \in \P_{1,2}(n)$ and $f$ is a coloring of $\pi$. A colored set partition $(\pi,f)$ is called a {\it partition of type D} if $f$ satisfies the following conditions \eqref{color1} and \eqref{color2}. 
\begin{enumerate}[(A)] 
\item\label{color1} Coloring of pair blocks $V\in\Pair(\pi)$ should start from the rightmost pair (= the pair having the largest left leg) and then go to left ones. 

\begin{enumerate}[\rm(1)]

\item\label{b} If $\sr(V)=\lcr(V)=\cover(V)=0$, then $V$ must be colored by $(-1)^{\npr(V)}$. For example take $\pi=\{\{1,3\},\{2,4\}\}$. If the block $\{2,4\}$ has color $-1$, then $V=\{1,3\}$ must be colored by $-1$. If the block $V=\{2,4\}$ has color 1 then the block $\{1,3\}$ must be colored by $1$. 

\item\label{a} Otherwise, $V$ can be colored by any of $-1$ and $1$. For example,  for $\pi=\{\{1,3\},\{2\}\}$ the block $\{1,3\}$ has a singleton $\{2\}$ on the right and so it can be colored by any of $-1$ and $1$.

\end{enumerate}
\item\label{color2} After coloring all pairs, we assign a unique color to each singleton $S$ as follows. 
\begin{enumerate}[\rm(1)]
\item\label{c} If $S$ is the rightmost singleton, then $S$ must be colored by $(-1)^{\np(\pi)}$. For example 
take $\pi=\{\{1,3\},\{2\}\}$. If $\{1,3\}$ is colored by $-1$ then $S=\{2\}$ must be colored by $-1$. If $\{1,3\}$ is colored by $1$ then $S=\{2\}$ must be colored by $1$. 

\item\label{d} Otherwise $S$ must be colored by $1$.

\end{enumerate}
\end{enumerate}
\end{definition}

We denote by $\PD_{1,2}(n)$ the set of all partitions of type D in $\P_{1,2}(n)$. We call $(\pi,f) \in \PD_{1,2}(n)$ a \emph{pair partition of $[n]$ of type D} if $\pi$ is a pair partition. The set of pair partitions of $[n]$ of type D is denoted by $\PD_2(n)$. Below we need colored partitions of a finite totally ordered set $T$, denoted e.g.\ by $\PD_{1,2}(T)$, which is defined naturally from the unique isomorphism $T\simeq [n]$ for some $n$. 

\begin{remark}\label{TD}
\begin{enumerate}[\rm(1)]
%\item Our definition of set partitions of type D looks different from \cite{R97}, but finding  

\item The set of type B partitions $\P_{1,2}^B(n)$ consists of colored partitions $(\pi,f)$ such that $\pi \in \P_{1,2}(n)$ and the singletons of $\pi$ must be colored by 1 (see \cite{BEH15}). In our definition, $\PD_{1,2}(n)$ is not a subset of $\P_{1,2}^B(n)$ since a partition of type D may have a negative singleton.  

\item About pair partitions, $\PD_{2}(n)$ is a subset of $\P^B_{2}(n)$ which is now the set of all colored pair partitions. 
The relationship between them can be written as follows. Take $\pi \in \P_2(n)$. Let $C_1,\dots,C_k$ be the outer connected components of $\pi$ and let 
\begin{align}
\pi_i=\{V\in \pi\mid\min(C_i)\leq \min(V)<\max(V)\leq  \max(C_i)\},
\label{eqatouterconnectedcomponents}
 \end{align} 
 which is the set of pairs contained in or covered by the outer connected component $C_i$ of $\pi$. Then $\pi=\bigcup_{i\in [k]}\pi_i$. The leftmost block $L_i$ of $\pi_i$ satisfies $\lcr(L_i)=\cover(L_i)=0$, and the other blocks $V$ of $\pi_i$ satisfy $\lcr(V)\neq0$ or $\cover(V)\neq0$. Hence, when we construct a coloring of $\pi$ of type D, we can choose arbitrary colors $\pm1$ of the blocks in $\pi_i\setminus\{L_i\}$ for each $i$, and we must choose a unique color of $L_i$. The definition \eqref{b} of type D colorings says that this unique color of $L$ is determined so that the number of negative blocks of $\pi_i$ becomes even (cf.\ Proposition \ref{property}\eqref{property1}). This argument gives, with the notation \eqref{eqatouterconnectedcomponents}, 
a characterization of pair partitions of type D: 
\begin{equation}\label{DefD}
\PD_{2}(n)=\{ (\pi,f) \in\P^B_2(n) \mid \textrm{$\np(\pi_i,f|_{\pi_i})$ is even for every $i$} \}. 
\end{equation}

\item In some sense, the above observation is also compatible with the relation between the Coxeter groups of type B and D. The Coxeter group of type B can be written as $B(n)=\Z_2^n \rtimes S(n) $ and hence it can be defined as all signed permutations of the numbers $\pm1,\dots,\pm n.$  The group $D(n)$ is a subgroup of $B(n)$ consisting of all signed permutations having an even number of negative entries in their window notation. Thus we can assign arbitrary signs $\pm1$  to the points $2,\dots,n$ and then we must assign a unique sign to the first element 1 depending on the signs of the other points $2,\dots,n.$ 
The construction of colorings of type D on each $\pi_i$ in \eqref{DefD} is exactly the above construction of $D(n) \subset B(n)$.  
\end{enumerate}
\end{remark}

\begin{example} 
Let $\pi=\{\{1\},\{2,5\},\{3\},\{4,6\}\} \in \P_{1,2}(6)$. Then the rightmost pair is $\{4,6\}$, which has the left crossing $\{2,5\}$, so it can be painted by any of $\pm1$. Then the next rightmost pair is $\{2,5\}$ which has the singleton $\{3\}$ on the right, so it can also be painted by $\pm1$. So there are four possible colorings of $\pi$. The singleton $\{1\}$ must be painted by $1$, and $\{3\}$ must be painted by either $1$ or $-1$ depending on $\np(\pi)$. 
\end{example}

\begin{example}\label{PD4} There are five pair partitions of $\{1, 2,3,4\}$ of type D. For example if $\pi=\{\{1,4\},\{2,3\}\}$, then in our terminology $V=\{2,3\}$ is on the right of $W=\{1,4\}$ and $\cover(V)=1$. According to \eqref{a}, $V$ can be painted by any of $1$ and $-1$. If $V$ is positive then \eqref{b} says that $W$ must be colored by $1$, and if $V$ is negative then $W$ must be colored by $-1$.

\setlength{\unitlength}{0.5cm}
%\vskip1cm
\begin{figure}[h]
\begin{center}
%\vspace{1cm}
\psset{nodesep=3pt}
\psset{angle=90}
\rnode{A}{\Node{1}} \rnode{B}{\Node{2}} \rnode{C}{\Node{3}} \rnode{D}{\Node{4}}  
\ncbar[armA=.4]{A}{B}
\put(-3,1.8){1}
\ncbar[armA=.4]{C}{D}
\put(-1.2,1.8){1}
\hspace{3mm}  
\rnode{A}{\Node{1}} \rnode{B}{\Node{2}}  \rnode{C}{\Node{3}}
\rnode{D}{\Node{4}}  
\ncbar[armA=.4]{A}{C}
\put(-3,1.8){1}
\ncbar[armA=.6]{B}{D}
\put(-1.5,2.2){1}
\hspace{3mm}  
\rnode{A}{\Node{1}} \rnode{B}{\Node{2}}  \rnode{C}{\Node{3}}
\rnode{D}{\Node{4}}   
\ncbar[armA=.4]{A}{C}
\put(-3.2,1.8){-1}
\ncbar[armA=.6]{B}{D}
\put(-1.5,2.2){-1}
\hspace{3mm}  
\rnode{A}{\Node{1}} \rnode{B}{\Node{2}}  \rnode{C}{\Node{3}}
\rnode{D}{\Node{4}}  
\ncbar[armA=1]{A}{D}
\put(-2.1,1.8){1}
\ncbar[armA=.4]{B}{C}
\put(-2.1,3){1} 
\hspace{3mm}  
\rnode{A}{\Node{1}} \rnode{B}{\Node{2}}  \rnode{C}{\Node{3}}
\rnode{D}{\Node{4}}  
\ncbar[armA=1]{A}{D}
\put(-2.4,1.8){-1}
\ncbar[armA=.4]{B}{C}
\put(-2.4,3){-1}
\end{center}
\vspace{-4mm}
\caption{$\PD_{2}(4)$.} 
\end{figure}
\end{example}

The colors of blocks of type D partitions must satisfy some properties. 

\begin{proposition} \label{property}
Suppose that $\pi \in \PD_{1,2}(n)$. 
\begin{enumerate}[\rm(1)]
\item\label{property1} If $\pi$ is a pair partition of type D, then $\np(\pi)$ is even (maybe zero). 
\item\label{property2} If $\pi$ contains the singleton $\{1\}$, then $\{1\}$ has color $1$.
\end{enumerate}
\end{proposition}
\begin{proof} The proofs can be given by contradiction. 

\eqref{property1} This is a statement weaker than \eqref{DefD}, but we give an independent proof. Suppose that the number of negative pairs is odd and consider the leftmost pair $V=\{1,i\}$ for some $i\in[n]$. Then $\lcr(V)=0$, $\cover(V)=0$, and the other pairs are on the right of $V$. If $V$ has color $1$ then $\npr(V)=\np(\pi)$ is odd, which contradicts \eqref{b} in the definition of partitions of type D. Otherwise, $V$ has color $-1$, then $\npr(V)=\np(\pi)-1$ is even, so we again get contradiction with the definition of partitions of type D.

\eqref{property2} Suppose that $\{1\}$ is colored by $-1$. The condition \eqref{color2} shows that $\{1\}$ must be the rightmost singleton, and so it is the unique singleton of $\pi$. The restriction $\pi|_{\{2,\dots,n\}}$ is also a (pair) partition of type D since removing the singleton $\{1\}$ does not matter in the condition \eqref{color1} for pairs. Then, what we have proved in \eqref{property1} says that $\pi|_{\{2,\dots,n\}}$ has an even number of negative pairs, which contradicts the assumption and \eqref{c}. 
\end{proof}

Type D partitions behave nicely with respect to the restriction of partitions on $\{2,\dots,n\}$. In fact they can be characterized by some recursion. 
\begin{lemma}\label{induction}
Suppose that $(\pi,f) \in \PD_{1,2}(\{2,\dots, n\})$. We get new partitions $(\tilde{\pi},\tilde{f}) \in \PD_{1,2}(n)$ from one of the following procedures. 
\begin{enumerate}[\rm(1)]
\item\label{proc1} Add the singleton $\{1\}$ to $(\pi,f)$ and color it by $1$. 
\item\label{proc2} If $\pi$ has singletons, choose a singleton $\{i\}$ of $\pi$, create a pair $\{1, i\}$, color it 
 and, if necessary, recolor the rightmost singleton of $\pi \setminus\{i\}$ so that we obtain a partition of $[n]$ of type D. More precisely: 
 \begin{enumerate}[\rm(i)]
\item\label{proc21}  If $\pi$ has at least two singletons then $\{1,i\}$ can be painted by any of $\pm1$. After choosing a color of $\{1,i\}$, the rightmost singleton of $\pi\setminus\{i\}$ must be repainted by $(-1)^{\np(\tilde{\pi})}$;   

\item\label{proc22} If $\pi$ has the unique singleton $\{i\}$ then $\{1,i\}$ must be colored by $(-1)^{\np(\pi)}$. 
\end{enumerate}
\end{enumerate}
When $(\pi,f)$ runs over $\PD_{1,2}(\{2,\dots, n\})$, every partition in $\PD_{1,2}(n)$ appears exactly once in the above procedures. 
\end{lemma}
\begin{proof} We first prove that a resulting partition $(\tilde{\pi}, \tilde{f})$ is indeed a partition of type D. 
In the procedure \eqref{proc1}, since the leftmost singleton does not matter in coloring the other pairs and singletons, so $(\tilde{\pi}, \tilde{f})$ is a partition of type D. 
In the procedure \eqref{proc2}, suppose that $\pi$ has a singleton at $i ~(\geq2)$, pairs $V_1,\dots, V_t$ on the left of $i$ and pairs $W_1,\dots, W_u$ on the strict right of $i$. Suppose that we take $i$ and create a pair $\{1,i\}$. The new pair $\{1,i\}$ is a left crossing or a covering of $V_1,\dots,V_t$, but not of $W_1,\dots, W_u$. Hence the old colors of $W_1,\dots, W_u$ given in $(\pi,f)$ satisfy the conditions \eqref{a},\eqref{b} in the new partition $(\tilde{\pi}, \tilde{f})$ too.  For $V_1,\dots,V_t$, they have the left crossing or covering $\{1,i\}$ and then according to \eqref{a} any color is allowed, so the old colors are valid. Thus the old colors of all pairs $V_1,\dots, V_t, W_1,\dots, W_u$ satisfy the required conditions \eqref{color1} on blocks of the type D partition $(\tilde{\pi}, \tilde{f})$. 

For singletons, the color of the rightmost singleton of $\tilde{\pi}$ is uniquely determined according to \eqref{d}. The other singletons of $\tilde{\pi}$ may keep the old colors unchanged. The above discussions show that $(\tilde{\pi}, \tilde{f})$ is a partition of type D. 

One can check that all the partitions $(\tilde{\pi}, \tilde{f})$ appearing in the above are distinct. Thus we only need to show that all partitions in $\PD_{1,2}(n)$ appear. Given a partition $(\sigma,g) \in \PD_{1,2}(n)$, the first case is that $1$ is a singleton of $\sigma$. Note then that $\{1\}$ has color $1$ by Proposition \ref{property}. In this case the restriction $(\pi,f):=(\sigma\setminus\{1\},g|_\pi)$ is a partition of $\{2,\dots, n\}$ of type D, and then after procedure \eqref{proc1} we get $(\tilde{\pi},\tilde{f})=(\sigma,g)$.  
The second case is that $\{1,i\}$ is a pair in $\sigma$ for some point $i$. The block $\{1,i\}$ may cover or cross (from the left) other pairs, say $V_1,\dots, V_t$. Let $\pi$ be the restriction of $\sigma$ to $\{2,\dots,n\}$, namely only $\{1,i\}$ is replaced by $\{i\}$. Then $V_1,\dots, V_t \in \pi$ have the singleton $i$ on the right, so the original colors of them satisfy the conditions for type D partitions of $\{2,\dots, n\}$. Keep the original colors of the other pairs too, and (re)color the rightmost singleton (if any) uniquely to get a partition $(\pi,f)$ of type D. Now we can revert this procedure from $(\pi,f)$ to $(\tilde{\pi},\tilde{f}) = (\sigma,g)$. 
\end{proof}

\subsection{Wick's formula of type D}\label{sec3.4}

Given $\epsilon=(\epsilon(1), \dots, \epsilon(n))\in\{1,\ast\}^n$, a set partition $\pi\in \P_{1,2}(n)$, written in the form
\begin{align*}
&\pi =\{\{i_1<j_1\}, \dots, \{i_k<j_k\}, \{s_1\}, \dots, \{s_m\}\}, \quad k,m \in \N \cup\{0\}, 
\end{align*}
is said to be {\it $\e$-compatible} if $\epsilon(i_p)=1 $ and $\epsilon(j_p)=\ast$ for all $1 \leq p \leq k$ and $\epsilon(s_p)=\ast$ for all $1 \leq p \leq m$. The set of $\e$-compatible partitions in $\P_{1,2}(n)$ is denoted by $\P_{1,2;\e}(n)$. We also let $\P_{2;\epsilon}(n):=\P_{1,2;\epsilon}(n)\cap \P_2(n)$. Let $\PD_{1,2;\epsilon}(n)$ be the set of type D partitions $(\pi,f)\in\PD_{1,2}(n)$ such that $\pi$ is $\e$-compatible, and similarly $\PD_{2;\epsilon}(n)$ be defined.  

Note that Lemma \ref{induction} can be extended to $\PD_{1,2;\e}(n)$. More precisely, given $\e=(\e(2), \dots, \e(n)) \in \{1,\ast\}^{n-1}$ and $(\pi,f) \in \PD_{1,2;\e|_{\{2,\dots,n\}}}(\{2,\dots,n\})$, the procedure \eqref{proc1} in Lemma \ref{induction} yields a colored partition $(\tilde{\pi},\tilde{f}) \in \PD_{1,2;(\ast,\e)}(n)$, and every colored partition of $\PD_{1,2;(\ast,\e)}(n)$ appears in this way. Similarly, the procedure \eqref{proc2} yields a colored partition in $\PD_{1,2;(1,\e)}(n)$ and every colored partition in $\PD_{1,2;(1,\e)}(n)$ appears in this way.

We establish a vector-version of Wick's formula.

\begin{theorem}\label{main theorem}   
For any $x_1,\dots,x_n \in H$ and any $\epsilon=(\epsilon(1), \dots, \epsilon(n))\in\{1,\ast\}^n$, we have
\begin{equation*}
\begin{split}
&\D^{\epsilon(1)}(x_{1})\cdots \D^{\epsilon(n)}(x_n)\Omega \\
&\qquad= \sum_{(\pi,f)\in\PD_{1,2;\epsilon}(n)}  q^{\Cr(\pi)+ \cs(\pi)+2 \cnp(\pi,f) + 2\npssr(\pi,f)} \\
&\qquad\quad \times \prod_{\substack{\{i<j\} \in \Pair(\pi)\\ f(\{i<j\})=1 } }\langle x_i, x_j\rangle \prod_{\substack{\{i<j\} \in \Pair(\pi)\\ f(\{i<j\})= -1}} \langle x_i, \overline{x}_j\rangle 
\bigotimes_{i\in \Sing(\pi,f)} x_i, 
\end{split}
\end{equation*} 
where $\bigotimes_{i\in \Sing(\pi,f)} x_i$ denotes the tensor product $x_{v_1}\otimes \cdots \otimes x_{v_m}$ if $\Sing(\pi,f)=\{v_1 < \cdots < v_m\}$ does not have a negative singleton, and $x_{v_1}\otimes \cdots \otimes \overline{x}_{v_m}$ if the rightmost singleton $v_m$ is negative. In the above
formula, we use the convention  $\bigotimes_{i\in \emptyset}x_i=\Omega.$ 
\end{theorem}
\begin{remark}
\begin{enumerate}[\rm(1)]
\item Usually people working on Fock spaces take $x_1,\dots, x_n$ from a real Hilbert subspace of $H$, and then the order $i<j$ does not matter since $\langle x_i, x_j\rangle = \langle x_j,x_i\rangle$. Our version is more general in this sense. 
\item If $\#\{i\in \{j,j+1,\dots,n\} \mid \epsilon(i)=1\}>\#\{i\in\{j,j+1,\dots,n\} \mid \epsilon(i)=\ast\}$ for some $j\in[n]$, then $\D^{\epsilon(1)}(x_{1})\cdots \D^{\epsilon(n)}(x_n)\Omega=0$. This case is also covered by Theorem \ref{main theorem} if we understand the sum over the empty set is 0 since $\PD_{1,2;\epsilon}(n)=\emptyset$ in this case. 
\end{enumerate}
\end{remark}

\begin{proof} 
The proof is given by induction and is based on Lemma \ref{induction}. When $n=1$,  $\D(x_1)\Omega=0$ and $\D^\ast(x_1)\Omega=x_1$ and hence the formula is true. Suppose that the formula is true for $n-1$. Then for any $(\epsilon(2),\cdots,\epsilon(n))\in\{1,\ast\}^{n-1}$, we get  
\begin{equation}
\begin{split}
&\D^{\epsilon(2)}(x_{2})\cdots \D^{\epsilon(n)}(x_n)\Omega \\
&\qquad= \sum_{(\pi,f)\in\PD_{1,2;\epsilon}(\{2,\dots,n\})}  q^{\Cr(\pi)+ \cs(\pi)+2 \cnp(\pi,f) + 2\npssr(\pi,f)} \\
&\qquad\quad \times \prod_{\substack{\{i<j\} \in \Pair(\pi)\\ f(\{i<j\})=1 } }\langle x_i, x_j\rangle \prod_{\substack{\{i<j\} \in \Pair(\pi)\\ f(\{i<j\})= -1}} \langle x_i, \overline{x}_j\rangle 
\bigotimes_{k\in \Sing(\pi,f)} x_k, 
\end{split} \label{formulaindution}
\end{equation}
We apply the operator $\D^{\epsilon(1)}(x_{1})$, which equals $\D^\ast(x_{1})$ if $\epsilon(1)=\ast$ and $\la_q(x_1)+ J \ra_{q}(\overline{x}_1)q^{N-1}$ if $\epsilon(1)=1$ from Theorem \ref{thm1}.

Case I: $\epsilon(1)=\ast$. The operator $\D^\ast(x_{1})$ creates a tensor component $x_{1}$ on the left. In terms of partitions, this corresponds to procedure \eqref{proc1} in Lemma \ref{induction}: to add the singleton $\{1\}$ (with color 1) to $(\pi,f)\in\PD_{1,2;\epsilon}(\{2,\dots,n\})$, to yield the new type D partition $(\tilde{\pi},\tilde{f})\in\PD_{1,2;\epsilon}(n)$. This map $(\pi,f)\mapsto (\tilde{\pi},\tilde{f})$ does not change the numbers $ \Cr, \cnp, \npssr$ or $\cs$, which is compatible with the fact that the action of $\D^\ast(x_{1})$ does not change the coefficient. 
Note that if $\epsilon(1)=\ast$, then any partition in $\PD_{1,2;\epsilon}(n)$ has the singleton $\{1\}$. Hence Theorem \ref{main theorem} holds for  $n$ and $\epsilon(1)=\ast$.

Case II: $\epsilon(1)=1$. Fix $(\pi,f)\in \PD_{1,2;\epsilon}(\{2,\cdots,n\})$ and suppose that $\pi$ has singletons  $k_1<\cdots <k_p <i < m_1 <\cdots <m_r$, negative pair blocks $V_1,\dots, V_t$ on the strict left of $i$ and pair blocks $W_1,\dots, W_u$ which cover $i$. There may be pair blocks on the strict right of $i$ or positive pair blocks on the strict left of $i$, but they do not matter. Then we discuss three cases separately: (i) $(p,r)\neq (0,0)$, (ii) $(p,r)= (0,0)$ and (iii) $\pi$ does not have a singleton.

\vspace{5mm}

Case II(i): $(p,r)\neq (0,0)$. In this situation we have at least one singleton other than the singleton $i$. 
In equation \eqref{formulaindution} we have two situations  
$$
\bigotimes_{k\in \Sing(\pi,f)} x_k =x_{k_1}\otimes \cdots \otimes x_{m_r} \textrm{ ~~~and ~~~} x_{k_1}\otimes \cdots \otimes \overline{x}_{m_r},
$$
and discuss these two cases separately.

Case II(i)1: Suppose that $\np(\pi)$ is even,  or equivalently, all singletons are colored by $1$ (see \eqref{color2} in the definition of partitions of type D). In this case 
$$
\bigotimes_{k\in \Sing(\pi,f)} x_k =x_{k_1}\otimes \cdots \otimes x_{k_p}\otimes x_{i}\otimes x_{m_1}\otimes \cdots\otimes x_{m_r}.  
$$
We discuss the left and right annihilation operators separately. 

Case II(i)1(a): The left action. The $q$-derivative $\la_q(x_{1})$ creates new $p+r+1$ terms.  %appear by using \eqref{rq}. 
In the $i^{\rm th}$ term the inner product $\langle x_{1},x_i\rangle$ appears with coefficient $q^p$. In terms of partitions this corresponds to a case of Lemma \ref{induction}\eqref{proc21}: to get a set partition $(\tilde{\pi}, \tilde{f}) \in \PD_{1,2;\epsilon}(n)$ by adding the pair $\{1, i\}$ with color $1$ to $\pi\setminus\{i\}$. This pair crosses the blocks $W_1, \dots, W_u$ and so increases the crossing number by $u$ but decreases the number of covered singletons by $u$. The new covered singletons $\{k_1\}, \dots, \{k_p\}$ and new inner negative blocks $V_1,\dots, V_t$ appear. Because $i$ is not a singleton in $(\tilde{\pi}, \tilde{f})$, the number of singletons on the strict right of negative blocks decreases by $t$. Altogether we have: 
$\Cr(\tilde{\pi})=\Cr(\pi)+u$, $\cs(\tilde{\pi})=\cs(\pi)-u+p$, $\cnp(\tilde{\pi},\tilde{f})=\cnp(\pi,f)+t$ and $\npssr(\tilde{\pi}, \tilde{f})= \npssr(\pi,f)-t$. So the exponent of $q$ increases by $p$. This factor $q^p$ is exactly the factor appearing in the $q$-derivative formula \eqref{rq} when $\la_q(x_{1})$ acts on $x_i$, see Figure \ref{fig:FiguraExemple1}. 
\begin{figure}[h]
\begin{center}
  \begin{tikzpicture}[thick,font=\small]
    \path % (-1,-0.1) node[circle] (a1) {$\dots$}
             (0,0) node[circle,,fill=blue!20,draw] (a) {-}
%          (0,-0.1) node[circle] (g) {$\dots$}
          (1,0) node[circle,draw] (b) {-}
          (2,0) node[circle,draw] (c) {-}
           (0,-0.7) node[] (c1) {$1$-th}
          (3,0) node[circle,draw] (d) {-}
          (8.4,2) node (fdf) {$\{1, i\}$ colored by  1}
%           \node (n6) at (1,3) {$\pi_{n-2}$};
          (4,0) node[circle,draw] (d1) {$\ast$}
          (4,-0.7) node[] (e3) {$k_1$-th}
          (5.1,-0.1) node[circle] (g) {$\dots$}
           (6,0) node[circle,draw] (g1) {$\ast$}
          (7,0) node[circle,draw] (g2) {$\ast$}
          (6,-0.7) node[] (e3) {$k_p$-th}
          (8,0) node[circle,fill=blue!20,draw] (h) {$\ast$}
         (8,-0.7) node[] (h2) {$i$-th}
         (9,0) node[circle,draw] (h1) {$\ast$}
          (9,-0.7) node[] (e3) {$m_1$-th}
            (11,1) node[] (e3) {$1$}
         (11,0) node[circle,draw] (e) {$\ast$}
          (10.1,-0.1) node[circle] (g) {$\dots$}
          (11,0) node[circle,draw] (e) {$\ast$}
 (12,0) node[circle,draw] (e1) {$\ast$}
 (13,0) node[circle,draw] (e2) {$\ast$}
 (11,-0.7) node[] (e3) {$m_r$-th}
         (5.7,-1) node[] (df) {};
    \draw (d) -- +(0,1) -| (g2);
    \draw (a) -- +(0,1.75) -| (h);
   \draw (d1) -- +(0,0.75) -| (d1);
    \draw (c)  --+(0,1.25) -| (e1);
    %\draw (d) -- +(0,0.75) -| (d);
    \draw (g1) -- +(0,0.75) -| (g1);
 \draw (e) -- +(0,0.75) -| (e);
 \draw (h1) -- +(0,0.75) -| (h1);
    \draw (b) -- +(0,1.5) -| (e2);
  \end{tikzpicture}
\end{center}
\caption{The visualization of the
action $\la_q(x_{1})$  on the  tensor product $x_{k_1}\otimes \cdots \otimes x_{m_r} $.}
\label{fig:FiguraExemple1}
\end{figure}

Case II(i)1(b). When $J\ra_{q}(\overline{x}_{1})q^{N-1}$ acts on the tensor, then new $p+r+1$ terms appear by using the right $q$-derivative formula \eqref{lq}. The $i^{\rm th}$ term has the coefficient $ q^{r+(p+r)} \langle\overline{x}_{1},x_i\rangle$. 
In terms of partitions, this corresponds to a case of Lemma \ref{induction}\eqref{proc21}: to create the new pair $\{1,i\}$ with color $-1$ and repaint the rightmost singleton $m_r$ or $k_p$ by $-1$ because now the number of negative pairs of the new partition $\tilde{\pi}$ is odd -- see Figure \ref{fig:FiguraExemple2}. Similarly to Step 1(a), we count the change of numbers and get 
$\Cr(\tilde{\pi})=\Cr(\pi)+u$, $\cs(\tilde{\pi})=\cs(\pi)-u+p$, $\cnp(\tilde{\pi},\tilde{f})=\cnp(\pi,f)+t$ and $\npssr(\tilde{\pi}, \tilde{f})= \npssr(\pi,f)-t+r$. Altogether, when moving from $(\pi,f)$ to $(\tilde{\pi}, \tilde{f})$, a new factor $q^{2r+p}\langle x_1, \overline x_i\rangle$ appears, which coincides with the coefficient appearing in the action of $J\ra_{q}(\overline{x}_{1})q^{N-1}$. 

\begin{figure}
\begin{center}
  \begin{tikzpicture}[thick,font=\small]
    \path              (0,0) node[circle,,fill=blue!20,draw] (a) {-}
          (1,0) node[circle,draw] (b) {-}
          (2,0) node[circle,draw] (c) {-}
           (0,-0.7) node[] (c1) {$1$-th}
          (3,0) node[circle,draw] (d) {-}
          (8.4,2) node (fdf) {$\{1, i\}$ colored by  -1}
          (4,0) node[circle,draw] (d1) {$\ast$}
          (4,-0.7) node[] (e3) {$k_1$-th}
          (5.1,-0.1) node[circle] (g) {$\dots$}
           (6,0) node[circle,draw] (g1) {$\ast$}
          (7,0) node[circle,draw] (g2) {$\ast$}
          (6,-0.7) node[] (e3) {$k_p$-th}
          (8,0) node[circle,fill=blue!20,draw] (h) {$\ast$}
         (8,-0.7) node[] (h2) {$i$-th}
         (9,0) node[circle,draw] (h1) {$\ast$}
          (9,-0.7) node[] (e3) {$m_1$-th}
         (11,0) node[circle,draw] (e) {$\ast$}
          (10.1,-0.1) node[circle] (g) {$\dots$}
          (11,0) node[circle,draw] (e) {$\ast$}
 (12,0) node[circle,draw] (e1) {$\ast$}
 (13,0) node[circle,draw] (e2) {$\ast$}
 (11,-0.7) node[] (e3) {$m_r$-th}
    (10.9,1) node (fdf) {  -1}
         (5.7,-1) node[] (df) {};
    \draw (d) -- +(0,1) -| (g2);
    \draw (a) -- +(0,1.75) -| (h);
   \draw (d1) -- +(0,0.75) -| (d1);
    \draw (c)  --+(0,1.25) -| (e1);
     \draw (g1) -- +(0,0.75) -| (g1);
 \draw (e) -- +(0,0.75) -| (e);
 \draw (h1) -- +(0,0.75) -| (h1);
    \draw (b) -- +(0,1.5) -| (e2);
  \end{tikzpicture}
\end{center}
\caption{The visualization of the
action $\ra_q(x_{1})$  on the  tensor product $x_{k_1}\otimes \cdots \otimes x_{m_r}$.}
\label{fig:FiguraExemple2}
\end{figure}

Case II(i)2: Suppose that $\np(\pi)$ is odd, or equivalently, the rightmost singleton is negative. This means that 
$$
\bigotimes_{k\in \Sing(\pi,f)} x_k =x_{k_1}\otimes \cdots \otimes x_{k_p}\otimes x_{i}\otimes x_{m_1}\otimes \cdots\otimes \overline{x}_{m_r} .
$$

Case II(i)2(a): The left action. The action of $\la_q(x_{1})$ creates $p+r+1$ terms. Two situations are possible. 
\begin{itemize}
\item If $r>0$ then we have a situation similar to Step 1(a) i.e.\ the $i^{\rm th}$ term has  the new coefficient $q^p\langle x_{1},x_i\rangle$. In terms of partitions this corresponds to a case of Lemma \ref{induction}\eqref{proc21}: we create the positive pair $\{1,i\}$ and then $\np(\tilde{\pi})$ is still odd so $m_r$ is still colored by $-1$. 

\item If $r=0$  then the $i^{\rm th}$ term creates the new factor $q^p\langle x_{1},\overline{x}_i\rangle$. In terms of partitions this corresponds to a case of Lemma \ref{induction}\eqref{proc21}: we create the negative pair $\{1, i\}$, so now $\np(\tilde{\pi})$ is even -- see Figure \ref{fig:FiguraExemple3}(a). The rightmost singleton $k_p$ of $\tilde \pi$ still has color 1.  
\end{itemize}
\begin{figure}[h]
\begin{center}
 \begin{tikzpicture}[thick,font=\small]
    \path   (4,3) node[circle] (a123) {(a)}
              (1,0) node[circle,fill=blue!20,draw] (c) {-}
      (0.9,-0.7) node[] (c1) {$1$-th}
          (2,0) node[circle,draw] (d) {-}
       (5.1,-0.7) node[] (c1) {$i$-th}
          (5.2,1.7) node (fdf) {Color -1}
         (4,1) node (fdf2) {Color 1}
          (3.1,-0.1) node[circle] (g) {$\dots$}
          (4,0) node[circle,draw] (g2) {$\ast$}
          (5,0) node[circle,fill=blue!20,draw] (h) {$\ast$}
    (6,0) node[circle,draw] (h2) {$\ast$}
         (3.8,-1.1) node[] (df) { };
    \draw (g2)  +(0,0.8) -| (g2);
    \draw (c)  --+(0,1.4) -| (h);
    \draw (d) -- +(0,2) -| (h2);
  \end{tikzpicture}
  \hspace{10mm}
 \begin{tikzpicture}[thick,font=\small]
    \path   (4,3) node[circle] (a123) {(b)}
          (1,0) node[circle,fill=blue!20,draw] (c) {-}
      (0.9,-0.7) node[] (c1) {$1$-th}
          (2,0) node[circle,draw] (d) {-}
 (5.1,-0.7) node[] (c1) {$i$-th}
 (4,-0.7) node[] (c1) {$k_p$-th}
          (5.2,1.7) node (fdf) {Color 1}
         (4,1) node (fdf2) {Color -1}
          (3.1,-0.1) node[circle] (g) {$\dots$}
          (4,0) node[circle,draw] (g2) {$\ast$}
          (5,0) node[circle,fill=blue!20,draw] (h) {$\ast$}
    (6,0) node[circle,draw] (h2) {$\ast$}
         (3.8,-1.1) node[] (df) { };
       \draw (g2)  +(0,0.8) -| (g2);
    \draw (c)  --+(0,1.4) -| (h);
    \draw (d) -- +(0,2) -| (h2);
  \end{tikzpicture}
\end{center}
\caption{The main structure of $(\tilde{\pi},\tilde{f})$ in Steps 2(a) and 2(b) when $r=0$.}
\label{fig:FiguraExemple3}
\end{figure}
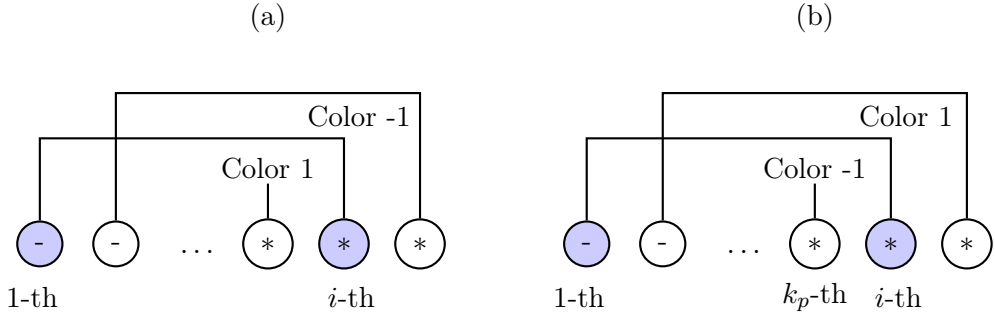

In both situations the new partition $(\tilde{\pi}, \tilde{f}) \in \PD_{1,2;\epsilon}(n)$ satisfies $\Cr(\tilde{\pi})=\Cr(\pi)+u$, $\cs(\tilde{\pi})=\cs(\pi)-u+p$, $\cnp(\tilde{\pi},\tilde{f})=\cnp(\pi,f)+t$ and $\npssr(\tilde{\pi}, \tilde{f})= \npssr(\pi,f)-t$, so the exponent of $q$ increases by $p$.  

Case II(i)2(b): The right action.  The operator $J\ra_{q}(\overline{x}_{1})q^{N-1}$ acting on the tensor product creates new $p+r+1$ terms. We have two situations. 
\begin{itemize}
\item If $r>0$ then the $i^{\rm th}$ term creates the new factor $q^{r+(p+r)}\langle x_1, \overline x_i\rangle$. Pictorially this corresponds to Lemma \ref{induction}\eqref{proc21}: the new negative pair $\{1,i\}$ is created and the singleton $m_r$ is repainted by $1$ because now $\np(\tilde{\pi})$ is even.
\item If $r=0$  then the $i^{\rm th}$ term has the new coefficient $q^p\langle \overline{x}_{1}, \overline{x}_i\rangle=q^p\langle x_{1}, x_i\rangle$. In terms of partitions  the new pair $\{1,i\}$ is created with color $1$, so the number of negative pairs is still odd and we must repaint the singleton $k_{p}$ by $-1$ -- see Figure \ref{fig:FiguraExemple3}(b). This is compatible with the fact that under this action we obtain $x_{k_1}\otimes \cdots \otimes \overline{x}_{k_p}$. 
\end{itemize}
Altogether, when moving from $(\pi,f)$ to $(\tilde{\pi}, \tilde{f})$, the exponent of $q$ increases by $2r+p$, which coincides with the coefficient appearing in the action of $J\ra_{q}(\overline{x}_{1})q^{N-1}$, creating the inner product $\langle  \overline{x}_{1},x_i\rangle$ if  $r>0$ and $\langle  x_{1},x_i\rangle$ if  $r=0$. Similarly to Case II(ii)1(b), we  get 
$\Cr(\tilde{\pi})=\Cr(\pi)+u$, $\cs(\tilde{\pi})=\cs(\pi)-u+p$, $\cnp(\tilde{\pi},\tilde{f})=\cnp(\pi,f)+t$ and $\npssr(\tilde{\pi}, \tilde{f})= \npssr(\pi,f)-t+r$.

\vspace{5mm}

 Case II(ii): $\epsilon(1)=1$ and $p=r=0$. In this case the partition $(\pi,f)$ has a unique singleton, and hence $\bigotimes_{k\in \Sing(\pi,f)} x_k =x_i$ or $\overline{x}_{i}$, according to the color of the singleton. We recall that $J\ra_{q}(\overline{x}_{1})x_i=J\ra_{q}(\overline{x}_{1})\overline{x}_{i}=0$, and so we only need to discuss the left $q$-derivative $l_q(x_1)$. The action creates the pair $\{1,i\}$. We have two situations.
\begin{itemize}
\item  If $(\pi,f)$ has the unique positive singleton $i$, then we have an even number of negative pairs.  
The $i^{\rm th}$ term of $l_q(x_1)x_i$ is $\langle x_{1},x_i\rangle$. In terms of partitions, this corresponds to Lemma \ref{induction}\eqref{proc22}: we create the pair $\{1, i\}$ with color $1$ to get a new pair partition $(\tilde{\pi},\tilde{f})$. 

\item  If $(\pi,f)$ has the unique negative singleton $i$, then we have an odd number of negative pairs.  
The $i^{\rm th}$ term of $l_q(x_1)\overline{x}_i$ is $\langle x_{1}, \overline{x}_i\rangle$. In terms of partitions, this corresponds to Lemma \ref{induction}\eqref{proc22}: we create the negative pair $\{1, i\}$ to get a new pair partition $(\tilde{\pi},\tilde{f})$. We emphasize that the number of negative pairs in this new partition is even, which is compatible with Proposition \ref{property} part (1).
\end{itemize}
Altogether, this pair creates a new partition $(\tilde{\pi}, \tilde{f})$ but it does not change the number $\Cr+ \cs+2 \cnp + 2\npssr$ by the same argument as in Case II(i)1(a).  This is compatible with the fact that the action $\la_q(x_{1})$ does not change the exponent of $q$. 

\vspace{5mm}

Case II(iii): $\pi$ does not have an singleton. Then the action of $d_q(x_1)$ on the vacuum gives zero. This is compatible with the observation that this situation does not appear in Lemma \ref{induction}, 
 i.e.\ we cannot pass from $ \PD_{2}(\{2,\dots, n\})$ (= vacuum vector) to $ \PD_{1,2}(n)$ by creating a new pair.%  if $\pi$.does not have an singleton.

Through Case II(i) -- Case II(iii) and by (the $\e$-compatible version of) Lemma \ref{induction}, we conclude that the action of $\D^{\e(1)}(x_1)$ yields all partitions in $\PD_{1,2;\epsilon}(n)$ with the desired coefficients. Hence we complete the proof.  
\end{proof}

\begin{example} We have the formula
\begin{align*} &\D(x_1)\D(x_2)\D^{\ast}(x_3)\D^{\ast}(x_4)\Omega\\&=q\langle x_{1},x_3\rangle\langle x_{2},x_4\rangle+q\langle  x_{1},\overline x_3\rangle\langle x_{2},\overline x_4\rangle+\langle x_{1},x_4\rangle\langle x_{2},x_3\rangle+q^2\langle x_{1},\overline x_4\rangle\langle x_{2},\overline x_3\rangle,\end{align*}
which corresponds to the colored partitions in Figure \ref{fig:FiguraExempleDDD*D*}. 
\setlength{\unitlength}{0.5cm}
%\vskip1cm
\begin{figure}[h]
\begin{center}
%\vspace{1cm}
\psset{nodesep=3pt}
\psset{angle=90}

\rnode{A}{\Node{1}} \rnode{B}{\Node{2}}  \rnode{C}{\Node{3}}
\rnode{D}{\Node{4}}  
\ncbar[armA=.4]{A}{C}
\put(-3,1.8){1}
\ncbar[armA=.6]{B}{D}
\put(-1.5,2.2){1}
\hspace{3mm}  
\rnode{A}{\Node{1}} \rnode{B}{\Node{2}}  \rnode{C}{\Node{3}}
\rnode{D}{\Node{4}}   
\ncbar[armA=.4]{A}{C}
\put(-3.2,1.8){-1}
\ncbar[armA=.6]{B}{D}
\put(-1.5,2.2){-1}
\hspace{3mm}  
\rnode{A}{\Node{1}} \rnode{B}{\Node{2}}  \rnode{C}{\Node{3}}
\rnode{D}{\Node{4}}  
\ncbar[armA=1]{A}{D}
\put(-2.1,1.8){1}
\ncbar[armA=.4]{B}{C}
\put(-2.1,3){1} 
\hspace{3mm}  
\rnode{A}{\Node{1}} \rnode{B}{\Node{2}}  \rnode{C}{\Node{3}}
\rnode{D}{\Node{4}}  
\ncbar[armA=1]{A}{D}
\put(-2.4,1.8){-1}
\ncbar[armA=.4]{B}{C}
\put(-2.4,3){-1}
\end{center}
\vspace{-5mm}
\caption{Pair partitions of type D --  $\PD_{2;(1,1,\ast,\ast)}(4)$.} 
\label{fig:FiguraExempleDDD*D*}
\end{figure}
\end{example}

When the involution is identity, our vector-version of Wick's formula can be written in terms of set partitions (of type A). 

\begin{proposition}\label{main theorem2}
Suppose that $x_i = \overline{x}_i \in H$, $i=1,2,\dots,n$. Then  
\begin{align*}
&\D^{\epsilon(1)}(x_1)\cdots \D^{\epsilon(n)}(x_n)\Omega \notag\\
&=\sum_{\pi\in \P_{1,2;\epsilon}(n)} 2^{-\outsr(\pi)}q^{\Cr(\pi)+\cs(\pi)} \prod_{\substack{\{i<j\} \in \Pair(\pi)} }\left(\left(1 + q^{2 \cover(\{i<j\}) +2 \ssr(\{i<j\})}\right)\langle x_i,x_j\rangle \right)\bigotimes_{i\in \Sing(\pi)} x_i, 
\end{align*}
where $\bigotimes_{i\in \Sing(\pi)} x_i$ denotes $x_{i_1}\otimes \cdots \otimes x_{i_k}$ when $\Sing(\pi)=\{i_1 < \cdots < i_k\}$. 
 \end{proposition}
\begin{proof}
 Since we have $\overline{x}_i=x_i$, coloring of singletons is not important. 
Take $\pi \in \P_{1,2;\epsilon}(n)$ and suppose that $\Pair(\pi)=\{V_1,\dots, V_k\}$. Then 
\begin{equation} 
\begin{split}
&\prod_{i=1}^k \left(1 + q^{2 (\cover(V_i) + \ssr(V_i)  )}\right) \\
&=  
 \sum_{(n_1,\dots,n_k)\in\{1,-1\}^k}\prod_{i=1}^k  \left(q^{2 (\cover(V_i) + \ssr(V_i)  )}\right)^{\delta_{n_i,-1}} \\
&= 
 \sum_{(n_1,\dots,n_k)\in\{1,-1\}^k}  q^{2\sum_{i=1}^k (\cover(V_i) + \ssr(V_i)  )\delta_{n_i,-1}}\\ 
&= 
2^{\#\{V \in \Pair(\pi)\,\mid\, \lcr(V)\,=\,\cover(V)\,=\,\sr(V)\,=\,0\}} \sum_{f} q^{2\cnp(\pi,f) + 2\npssr(\pi,f) },
\end{split}
\end{equation}
where $f$ runs over all possible type D colorings of $\Pair(\pi)$.
The last formula follows from the
observation that a pair block $V \in \pi$ must be painted by a unique color ($1$ or $-1$) if and only if 
\begin{align}
\lcr(V)=\cover(V)= \sr(V)=0\label{Warunki1-1TypeD}
\end{align} 
by the definition of partitions of  type $D$. So the sum 
\begin{equation}
\sum_{(n_1,\dots,n_k)\in\{1,-1\}^k}  q^{2\sum_{i=1}^k (\cover(V_i) + \ssr(V_i)  )\delta_{n_i,-1}}
\end{equation}
counts the colorings of $V$ twice as many as the type D colorings of $V$ whenever $V$ satisfies \eqref{Warunki1-1TypeD}. 

From Theorem \ref{main theorem} it is sufficient to show that 
the number of pairs $V$ which satisfy condition \eqref{Warunki1-1TypeD} is equal to $\outsr(\pi)$, the number of outer connected components with size at least $2$ without singletons on the right. 
This follows from the fact that each connected component $C\in \outsr(\pi)$ with size $\geq2$ contains a unique block of $\pi$ which does not have a left crossing or a covering. In fact it is the block of $\pi$ having the minimal point of $C$.  
\end{proof}

The main theorem is the Wick formula of type D, which is similar to the type B case when $\alpha=1$, cf.\ \eqref{WickB}, but now we use completely different partitions. 

\begin{theorem}\label{thm2} 
Suppose that $x_1,\dots,x_n \in H$ and $\epsilon\in\{1,\ast\}^n$. 
\begin{enumerate}[\rm(1)] 
\item 
$\displaystyle
\langle\Omega, \D^{\epsilon(1)}(x_1)\cdots \D^{\epsilon(n)}(x_n)\Omega\rangle_{q}
=
\sum_{(\pi,f) \in \PD_{2;\epsilon}(n)} q^{\Cr(\pi)+2 \cnp(\pi,f)}\prod_{\substack{\{i<j\} \in \pi \\ f(\{i<j\})=1} }\langle x_i, x_j\rangle \prod_{\substack{\{i<j\} \in \pi\\ f(\{i<j\})=-1}} \langle x_i, \overline{x}_j\rangle. 
$
\item\label{GaussOp}
$\displaystyle
\langle\Omega, \G(x_1)\cdots \G(x_n)\Omega\rangle_{q}=
\displaystyle\sum_{(\pi,f) \in \PD_2(n)}q^{\Cr(\pi)+2 \cnp(\pi,f)}\prod_{\substack{\{i<j\} \in \pi \\ f(\{i<j\})=1} }\langle x_i, x_j\rangle \prod_{\substack{\{i<j\} \in \pi\\ f(\{i<j\})=-1}} \langle x_i,\overline{x}_j\rangle. 
$
\end{enumerate}
\end{theorem}
\begin{proof}
(1) is clear from Theorem \ref{main theorem} and (2) follows from (1) by taking the sum over all $\epsilon$. 
\end{proof}

\begin{corollary}\label{Wick1} Assume that $x_i = \overline{x}_i \in H$ for $i=1,\dots,2m$. In the limit $q\to1$, we recover the $t$-transformed classical Brownian motion \cite[Theorem 9.2]{BW01} with $t=2$: 
\begin{align*}
\langle\Omega, G_1(x_1)\cdots G_1(x_{2 m})\Omega\rangle_1
=\sum_{\pi \in \P_2(2 m)} 2^{m-\#\out(\pi)}\prod_{\{i<j\} \in\pi }\langle x_i, x_j\rangle. 
\end{align*}
\end{corollary}
\begin{proof}
The formula follows from Proposition \ref{main theorem2} and taking the sum over $\e$. 
\end{proof}

\begin{corollary}\label{free Wick}  Assume that $x_i \in H$ for $i=1,\dots,2m$. 
When $q=0$ we recover the moments for a semicircular system 
\begin{align*}
\langle\Omega, G_0(x_1)\cdots G_0(x_{2 m})\Omega\rangle_0
=\sum_{\pi \in \NC_2(2 m)} \prod_{\{i<j\} \in\pi }\langle x_i, x_j\rangle,  
\end{align*}
where $\NC_2(2m):= \{\pi \in \P_2(2m) \mid \Cr(\pi)=0\}$.  Note that the involution ${}^-$ does not appear in the formula. 
\end{corollary}
\begin{proof} 
The formula follows from Corollary \ref{thm2}\eqref{GaussOp}. Note that $0^{\Cr(\pi)+2\cnp(\pi,f)}$ gives a nonzero value (=1) only when $\Cr(\pi)+2 \cnp(\pi,f)=0$, which implies that all blocks are positive by the definition of type D partitions.  
\end{proof}

\subsection{Traciality of the vacuum state} 
In the context of the von Neumann algebra generated by a free Brownian motion ($q=0$ case), it is common to consider a Brownian motion indexed by a real Hilbert subspace; otherwise the vacuum state would not be a trace. We follow this strategy and assume that $H_\R$ is a real Hilbert subspace of $H$ such that $H=H_\R \oplus i H_\R$. 
When considering elements in $H_\R$, it holds true that $\langle x,y\rangle=\langle y,x\rangle$. 
Let $\A$ be the von Neumann algebra generated by $\{\G(x)\mid x\in H_\R\}$ acting on the completion of $\Fq$.

\begin{proposition}\label{NTrace} Let $q \in (-1,1)$.  Suppose that $\text{\rm dim}(H_\R) \geq2$. Then the vacuum state is a trace on $\A$  if and only if $q=0$. 
\end{proposition}
%\begin{remark}
%We need to assume that $x_i \in H_\R$ to use the property $\langle x_i, x_j \rangle = \langle x_j, x_i\rangle$. 
%\end{remark}
\begin{proof} Corollary \ref{thm2}\eqref{GaussOp} for $n=4$ reads (see Example \ref{PD4})
\begin{equation}
\begin{split}
&\langle\Omega, \G(x_1)\G(x_2)\G(x_3)\G(x_4)\Omega\rangle_{q} \\
&\qquad= 
\langle x_1, x_2\rangle\langle x_3, x_4\rangle 
+ q \langle x_1, x_3\rangle\langle x_2, x_4\rangle+ q \langle x_1, \overline{x}_3\rangle\langle x_2, \overline{x}_4\rangle \\
&\qquad\quad
+\langle x_1, x_4\rangle\langle x_2, x_3\rangle+ q^2\langle x_1, \overline{x}_4\rangle\langle x_2, \overline{x}_3\rangle 
\end{split}
\end{equation}
and by permuting $x_1,x_2,x_3,x_4$, 
\begin{equation}
\begin{split}
&\langle\Omega, \G(x_2)\G(x_3)\G(x_4)\G(x_1)\Omega\rangle_{q} \\
&\qquad= 
\langle x_2, x_3\rangle\langle x_1, x_4\rangle 
+ q \langle x_2, x_4\rangle\langle x_1, x_3\rangle + q \langle x_2, \overline{x}_4\rangle\langle x_1, \overline{x}_3\rangle \\
&\qquad\quad
+\langle x_1, x_2\rangle\langle x_3, x_4\rangle+ q^2\langle x_1, \overline{x}_2\rangle\langle x_3, \overline{x}_4\rangle, 
\end{split}
\end{equation}
where the assumption $x_1,\dots, x_4 \in H_\R$ was used. 
Hence 
\begin{equation}
\begin{split}
&\langle\Omega, \G(x_1)\G(x_2)\G(x_3)\G(x_4)\Omega\rangle_{q}-\langle\Omega, \G(x_2)\G(x_3)\G(x_4)\G(x_1)\Omega\rangle_{q} \\
&\qquad\qquad= q^2(\langle x_1, \overline{x}_4\rangle\langle x_2, \overline{x}_3\rangle-\langle x_1, \overline{x}_2\rangle\langle x_3, \overline{x}_4\rangle).
\end{split}
\end{equation}
Therefore the vacuum state is not a trace when $q \neq 0$, because 
when  \text{dim}$(H_\R) \geq 2$ there are two orthogonal unit eigenvectors $e_1,e_2$ of the involution ${}^-$, and we take $x_1=x_2=e_1$ and $x_3=x_4=e_2$. 
%Note that for some $s,t \in\{\pm1\}$ we have $\overline{e_1}= s e_1, \overline{e_3}= t e_3$. 
When $q=0$, the von Neumann algebra becomes the free von Neumann algebra and the traciality of the vacuum state is well known. It actually follows from the free Wick formula Corollary \ref{free Wick}.
\end{proof}

\begin{center}
Open problems
\end{center}

\begin{enumerate}[\rm(1)]

\item Study the von Neumann algebra $\A$, in particular, injectivity, completely bounded approximation property, cyclic separating property of the vacuum,  factoriality and type.  

\item Prove the existence of a classical Markov process realization of the Brownian motion of type D (see \cite{BKS97} for the type A case). 

\item Find a connection between noncrossing partitions of type D in \cite{AR04,R97} and our pair partitions of type D. This problem may be related to the problem of finding a ``free probability of type D'' in the spirit of Biane, Goodman and Nica \cite{BGN03}. 

\item Construct a Fock space deformed by affine Coxeter groups which are infinite groups.

\item Compute the explicit form of $\mu_q$, the distribution of the Gaussian operator of type D. Is it absolutely continuous with respect to the Lebesgue measure?  

\item Describe the distribution of $\G(x)$ when $x$ is not an eigenvector of the involution. 

\item Find the exact values of the norms of creation and the Gaussian operators of type D.

\end{enumerate}

\begin{center} Acknowledgments
\end{center}

The work was supported by the MAESTRO 
grant DEC-2011/02/A/ ST1/00119 (M.\ Bo\.zejko), Austrian Science Fund (FWF) Project No P 25510-N26 (W.\ Ejsmont), grant number 2014/15/B/ST1/00064 from the Narodowe
Centrum Nauki (W.\ Ejsmont), Wymian\k{e} osobow\k{a} z Austri\k{a} Project No DWM.ZWB.183.1.2016 (M.\ Bo\.zejko, W.\ Ejsmont)  and JSPS Grant-in-Aid for Young Scientists (B) 15K17549 
(T.\ Hasebe).

\end{document}